\documentclass[10pt,twoside]{siamart1116}
\usepackage[english]{babel}
\usepackage{graphicx,epstopdf,epsfig}
\usepackage{amsfonts,epsfig,fancyhdr,graphics, hyperref,amsmath,amssymb}

\usepackage{url, enumerate}
\usepackage[bb=boondox]{mathalfa} % for double struck 1 and 0
\usepackage[mathlines]{lineno}
\usepackage{mathtools}

\theoremstyle{definition}
\newtheorem{remark}[theorem]{Remark}
\theoremstyle{definition}
\newtheorem{example}[theorem]{Example}

\renewcommand{\theequation}{\thesection.\arabic{equation}}
\renewtheorem{theorem}{Theorem}[section]
\usepackage{tkz-graph}
\usepackage{tkz-berge}

\newtheorem{obs}[theorem]{Observation}

\theoremstyle{definition}

\numberwithin{figure}{section}  
\numberwithin{table}{section}   

\usepackage{color}
\definecolor{red}{rgb}{.8,0,0}

\definecolor{blu}{rgb}{0,0,1}

\def\mtx#1{\begin{bmatrix} #1 \end{bmatrix}}

\newcommand{\R}{{\mathbb R}}
\newcommand{\C}{{\mathbb C}}

\newcommand{\Cnn}{\C^{n\times n}}
\newcommand{\Rnn}{\R^{n\times n}}
\newcommand{\ba}{{\bf a}}
\newcommand{\bx}{{\bf x}}
\newcommand{\bb}{{\bf b}}
\newcommand{\bc}{{\bf c}}
\newcommand{\bv}{{\bf v}}

\newcommand{\by}{{\bf y}}
\newcommand{\bz}{{\bf z}}
\usepackage[bb=boondox]{mathalfa} % for double struck 1 and 0
\newcommand{\bone}{\mathbb{1}} % double-struck 1
\newcommand{\J}{\mathbb{J}}
\newcommand{\I}{\mathbb{I}} 
\newcommand{\boldzero}{\bf{0}} % double-struck 0

\newcommand{\D}{\mathcal D}
\newcommand{\dig}{\Gamma}

\newcommand{\DL}{{\D^L}}
\newcommand{\DQ}{{\D^Q}}
\newcommand{\spec}{\operatorname{spec}}
\newcommand{\dspec}{\operatorname{spec}_{\D}}
\newcommand{\dLspec}{\operatorname{spec}_{\DL}}
\newcommand{\dQspec}{\operatorname{spec}_{\DQ}}
\newcommand{\dev}{\partial}
\newcommand{\dlev}{\partial^L}
\newcommand{\dQev}{\partial^Q}

\newcommand{\diag}{\operatorname{diag}}
\newcommand{\diam}{\operatorname{diam}}

\newcommand{\A}{\mathcal A}
\newcommand{\rank}{\operatorname{rank}}

\newcommand{\cp}{\, \Box\,}
\newcommand{\OL}{\overline}

\newcommand{\trans}{t} %{\operatorname{trans}}
\newcommand{\mult}{\operatorname{mult}}
\newcommand{\gmult}{\operatorname{gmult}}
\newcommand{\JCF}{\operatorname{J}}

\newcommand{\directp}{\times}
\newcommand{\strongp}{\boxtimes}

\newcommand{\cpj}{\,{\scriptsize{\fbox{$\J$}}} \,}
\newcommand{\cpi}{\,{\scriptsize{\fbox{$\I$}}} \,}
\newcommand\lexp{\,{\textcircled{\scriptsize{L}}} \,}

\newcommand{\ds}{\displaystyle}
\newcommand{\bit}{\begin{itemize}}
\newcommand{\eit}{\end{itemize}}
\newcommand{\ben}{\begin{enumerate}}
\newcommand{\een}{\end{enumerate}}
\newcommand{\beq}{\begin{equation}}
\newcommand{\eeq}{\end{equation}}
\newcommand{\bea}{\begin{eqnarray*}}
\newcommand{\eea}{\end{eqnarray*}}
\newcommand{\bean}{\begin{eqnarray}}
\newcommand{\eean}{\end{eqnarray}}
\newcommand{\bpf}{\begin{proof}}
\newcommand{\epf}{\end{proof}}
\newcommand{\x}{\times}

\newcommand{\lam}{\lambda}

\newcommand{\lp}{\left(}
\newcommand{\rp}{\right)}
\newcommand{\lb}{\left[}
\newcommand{\rb}{\right]}

\begin{document}

\title{Spectral theory of products of digraphs}
\author{Minerva Catral\thanks{Department of Mathematics, Xavier
University, Cincinnati, OH 45207, USA (catralm@xavier.edu).}
\and
Lorenzo Ciardo\thanks{Department of Mathematics, University of Oslo, Oslo, NO 0316, Norway (lorenzci@math.uio.no).}
\and  Leslie
Hogben\thanks{Department of Mathematics, Iowa State University,
Ames, IA 50011, USA and American Institute of Mathematics, 600 E. Brokaw Road, San Jose, CA 95112, USA
(hogben@aimath.org).}\and
Carolyn Reinhart\thanks{Department of Mathematics, Iowa State University, Ames, IA 50011, USA (reinh196@iastate.edu).}
 }

\maketitle

% \linenumbers
 
%%%%%%%%%%%%%%%%%%%%%%%%%%%%%%%%%%%%%%%%%%%%%

\maketitle

%%%%%%%%%%%%%%%%%%%%%%%%%%%%%%%%%%%%%%%%%%%%
\begin{abstract}
A unified approach to the determination of eigenvalues and eigenvectors of specific matrices associated with directed graphs is presented.  Matrices studied include the distance matrix, distance Laplacian, and distance signless Laplacian, in addition to the adjacency matrix, Laplacian, and signless Laplacian.  Various sums of Kronecker products of nonnegative matrices are introduced to model the  Cartesian and lexicographic products of digraphs.  The Jordan canonical form is applied extensively to the analysis of spectra and eigenvectors.   The analysis shows that Cartesian products provide a method for building infinite families of transmission regular digraphs with few distinct distance eigenvalues.
\end{abstract}

\medskip


 \medskip\medskip
%%%%%%%%%%%%%%%%%%%%%%%%%%%%%%%%%%%%%%%%%%%%

\section{Introduction}\label{sintro}

{\em Spectral graph theory} has traditionally been the study of the relation between properties of (undirected) graphs and the spectrum of the adjacency matrix, Laplacian matrix, or signless Laplacian of the graph \cite{BH}. The distance matrix of a graph was introduced  in the study of a data communication problem \cite{GP71} and has attracted a lot of interest recently (see, e.g, \cite{AH14} for a survey on distance spectra of graphs).  Recently the distance Laplacian and distance signless Laplacian of a graph have been studied (see, for example, \cite{AH13}).  Spectral theory of digraphs is a developing area of research but  so far focused primarily on the spectral radius of the adjacency matrix  (see \cite{B10} for a survey on spectra of digraphs).

A graph $G = (V(G),E(G))$ consists of a finite set  $V(G) = \{v_1, \dots, v_n\}$ of vertices and a set  $E(G)$ of two-element subsets $\{v_i, v_j\}$ called {\em edges}.   A digraph $\dig = (V(\dig),E(\dig))$ consists of a finite set  $V(\dig) = \{v_1, \dots, v_n\}$ of vertices and a set  $E(\dig)$ of ordered pairs of distinct vertices $(v_i, v_j)$ called {\em arcs}. Observe that neither a graph nor digraph can have a loop (an edge or arc with the vertices equal).  For a digraph $\dig$ (respectively, graph $G$), a {\em dipath} (respectively, {\em path}) from $u$ to $v$ is a sequence of vertices and arcs (respectively, edges) 
$u=w_1, e_1=(w_1,w_2),w_2, e_2=(w_2,w_3),\dots,w_k, e_k=(w_k,w_{k+1}),w_{k+1}=v$ (in a path, the arcs are replaced by unordered edges). A digraph (or graph) of order at least two is {\em strongly connected} (or {\em connected}) if for every pair of vertices $u, v$, there is a dipath (or path) from $u$ to $v$.  %All digraphs (or graphs) considered in this paper are strongly connected (or connected) unless otherwise stated.

 The {\em adjacency matrix} of $\dig$  (or $G$), denoted by $\A(\dig)$ (or $\A(G)$), is the $n \times n$ matrix with $(i,j)$ entry equal to $1$ if  $(v_i, v_j)$ (or $\{v_i, v_j\}$) is an arc (or  edge) of $\dig$ (or $G$), and $0$ otherwise. The {\em Laplacian matrix} of $\dig$ (or $G$), denoted by $L(\dig)$ (or $L(G)$), is defined as $D(\dig) - \A(\dig)$ (or $D(G) - \A(G)$), where $D(\dig)$  (or $D(G)$) is the diagonal matrix having the $i$-th diagonal entry equal to the {\em out-degree} (or {\em degree}) of the vertex $v_i$, i.e., the number of arcs (or edges) starting at $v_i$.  The matrix $D(\dig) + \A(\dig)$ (or $D(G) + \A(G)$) is called the {\em signless Laplacian matrix} of  $\dig$ (or $G$) and is denoted by $Q(\dig)$ (or $Q(G)$). For a strongly connected digraph $\dig$ (or a connected graph $G$), the {\em distance matrix}, denoted $\D(\dig)$ (or $\D(G)$), is the $n \times n$ matrix with $(i,j)$ entry equal to $d(v_i,v_j)$,  the {\em distance} from  $v_i$ to $v_j$, i.e., the length of a shortest dipath (or path) from $v_i$ to $v_j$;  use of a distance matrix  implies the digraph (or graph) is strongly connected (or connected). 
 The {\em transmission} of vertex $v_i$ is defined as $\trans(v_i) =  \sum_{j=1}^n d(v_i,v_j)$. The transmission of a vertex in a digraph could have been called the out-transmission because it is the sum of the out-distances, i.e., the distances from $v_i$ to other vertices. 
 The {\em distance Laplacian} matrix and the {\em distance signless Laplacian} matrix, denoted by $\DL$ and $\DQ$, respectively, are defined by  $\DL(\dig) = T(\dig)  - \D(\dig)$  and $\DQ(\dig) = T(\dig) + \D(\dig)$, where $T(\dig)$  is the diagonal matrix with $\trans(v_i)$  as the $i$-th diagonal entry;   $\DL(G)$ and $\DQ(G)$ are defined analogously.  A digraph is {\em out-regular} or {\em $r$-out-regular} if every vertex has out-degree $r$.  A strongly connected digraph is {\em transmission regular} or {\em $t$-transmission regular}  if every vertex has transmission $t$.  The terms {\em regular}, {\em $r$-regular},  {\em transmission regular}, and {\em $t$-transmission regular} are defined analogously for graphs.

 For a real $n \times n$ matrix  $M$, the {\em algebraic multiplicity} $\mult_M(z)$ of a number $z\in\C$ with respect to $M$ is the number of times $(x-z)$ appears as a factor in the characteristic polynomial $p(x)$ of $M$, and the {\em geometric multiplicity}  $\gmult_M(z)$ is the dimension of the eigenspace $ES_M(z)$ of $M$ relative to $z$ ($\mult_M(z)=\gmult_M(z)=0$ if $z$ is not an eigenvalue of $M$).   The {\em spectrum} of $M$, denoted by $\spec(M)$,  is the multiset whose elements are the $n$ (complex) eigenvalues of $M$ (i.e., the number of times each eigenvalue appears in $\spec(M)$ is its algebraic multiplicity).  The spectrum is often written as $\spec(M)=\{\lam_1^{(m_1)},\dots,\lam_q^{(m_q)}\}$ where $\lam_1,\dots, \lam_q$ are the distinct eigenvalues of $M$ and $m_1,\dots,m_q$ are the (algebraic) multiplicities.

There are several spectra associated with a digraph $\dig$, namely, $\spec_\A(\dig)=\spec(\A(\dig))$ ({\em adjacency spectrum}), $\spec_L(\dig)=\spec(L(\dig))$ ({\em Laplacian spectrum}), $\spec_Q(\dig)=\spec(Q(\dig))$ ({\em signless Laplacian spectrum}), $\dspec(\dig)=\spec(\D(\dig))$ ({\em distance spectrum}), $\dLspec(\dig)=\spec(\DL(\dig))$ ({\em distance Laplacian spectrum}), and $\dQspec(\dig)=\spec(\DQ(\dig))$ ({\em distance signless Laplacian spectrum}). For a graph $G$, the relevant spectra are  $\spec_\A(G)=\spec(\A(G))$, $\spec_L(G)=\spec(L(G))$, $\spec_Q(G)=\spec(Q(G))$, $\dspec(G)=\spec(\D(G))$, $\dLspec(G)=\spec(\DL(G))$, and $\dQspec(G)=\spec(\DQ(G))$, with the same terminology.

This paper contributes to the study of the spectra of digraphs,  particularly by presenting new results on eigenvalues and eigenvectors of the distance matrix of various products of digraphs.   In Section \ref{scartprod}, we analyze constructions of matrices (sums of Kronecker products) that produce  the  adjacency and distance matrices. We use the Jordan canonical form to  derive formulas for the spectra of these constructions in terms of the spectra of the original matrices, and apply these results to determine the   adjacency spectrum of a Cartesian product of two digraphs  in terms of  the adjacency    spectra of the digraphs, and to determine the     distance  spectrum of a Cartesian product of two transmission regular digraphs  in terms of  the  distance spectra of the digraphs. These formulas show that Cartesian products provide a method for building infinite families of transmission regular digraphs with few distinct distance eigenvalues; this is discussed in Section \ref{sSRD}. In some cases we establish formulas for the Jordan canonical form, geometric multiplicities of eigenvalues, or eigenvectors  of the constructed matrix.    In Section \ref{slexprod}, we investigate the spectra of lexicographic products of digraphs by similar methods. Section \ref{sDirectStrongprod} gives a brief discussion on the spectra of the direct and strong products.

In the remainder of this introduction, we define various digraph products and the matrix constructions that describe  the matrices associated with these digraphs, and state elementary results we will use.

\subsection{Digraph products and matrix constructions}\label{s:mtx-prod}

Let $\dig$ and $\dig'$ be digraphs of orders $n$ and $n'$, respectively.  We consider the four standard associative digraph products, namely the {\em Cartesian product} $\dig \cp \dig'$, the {\em lexicographic product} $\dig \lexp \dig'$, the {\em direct product}  $\dig \directp \dig'$ and the {\em strong product} $\dig \strongp \dig'$ \cite{H18}. Each has vertex set $V(\dig) \times V(\dig')$ and their arc sets are:
\[
\begin{array}{lcl}
E(\dig \cp \dig') & = & \{ ((x,x'),(y,y') )\ | \  x'=y' \mbox{ and } (x,y) \in E(\dig), \mbox{ \bf or }  x=y \mbox{ and  } (x',y')  \in E(\dig') \},\\
E(\dig \lexp \dig') & = & \{ ((x,x'),(y,y') )\ | \  (x,y) \in E(\dig), \mbox{ \bf or }  x=y \mbox{ and  } (x',y')  \in E(\dig')  \},\\
E(\dig\directp\dig') & = & \{ ((x,x'),(y,y') )\ | \ (x,y)  \in E(\dig) \mbox{ and } (x',y') \in E(\dig') \}, \mbox{ and }\\ 
E(\dig\strongp \dig') & =  & E(\dig\cp\dig') \cup E(\dig\directp\dig').
\end{array}
\]

Rather than  establishing spectral results just for the matrices associated with these digraph products, we develop a general theory of the spectra of matrices constructed in a specified form  as a sum of Kronecker products of matrices with the identity or with the all ones matrix.
 The {\em Kronecker product} of an $n\x n$ matrix $A=[a_{ij}]$ and a $n'\x n'$ matrix $A'$ denoted by $A\otimes A'$,
is the $nn'\x nn'$ block matrix
\[A\otimes A'=
\mtx{
a_{11}A'&a_{12}A'&\cdots&a_{1n}A'\\
a_{21}A'&a_{22}A'&\cdots&a_{2n}A'\\
 \vdots&\vdots&\ddots&\vdots\\
 a_{n1}A'&a_{n2}A'&\cdots&a_{nn}A'}.\]
 Let $M\in\Cnn$ and $M'\in\C^{n'\times n'}$. We use the following notation: The $n\x n$ identity matrix is denoted by $\I_n$. The $n\x n$ all ones matrix  is denoted by $\J_n$. The all ones $n$-vector is denoted by $\bone$. The  all zeros matrix  is denoted by $O$. The all zeros vector is denoted by $\boldzero$. Define the matrix constructions 
 \[M\cpi M'=M\otimes \I_{n'}+\I_n\otimes M'\in\mathbb{C}^{(nn')\times (nn')},\] 
\[M\cpj M'=M\otimes \J_{n'}+\J_n\otimes M'\in\mathbb{C}^{(nn')\times (nn')},\] 
and
\[M\lexp M'=M\otimes \J_{n'}+\I_n\otimes M'\in\mathbb{C}^{(nn')\times (nn')}.\] 
%If $M$ and $M'$ are real and symmetric, then formulas for the spectra of $M\cpi M'$ and $M\cpj M'$ in terms of those of $M$ and $M'$ are well known. In section \ref{scartprod},  we use the Jordan canonical form to prove results about the spectra and then also determine information about eigenvectors.  The results apply to adjacency and distance matrices of digraphs.   
Then, as in the case with graphs,  
\[\A(\dig\cp\dig')=\A(\dig) \cpi \A(\dig') 
 \mbox{ \  and \  }
 \D(\dig\cp\dig')=\D(\dig)\cpj \D(\dig').\]
 The matrix construction $M\lexp M'$  arises naturally for the adjacency matrix of the lexicographic product, because $\A(\dig\lexp\dig')=\A(\dig) \lexp \A(\dig')$ (as is the case for graphs), and has some uses for the distance matrix   $\D(\dig\lexp\dig')$, as discussed in Section  \ref{slexprod}  (in particular, see Observation \ref{obs:lex-longcycle}).
% Also,  as in the case with graphs,   $\A(\dig\lexp\dig')=\A(\dig) \lexp \A(\dig')$,  and the matrix product $M\lexp M'$  has some uses for the distance matrix of the lexicographic product $\dig \lexp \dig'$,

For many cases, we determine the spectrum of the construction of $M$ and $M'$ by using the construction of Jordan canonical forms of $M$ and $M'$ to obtain a triangular matrix that is similar to the construction of $M$ and $M'$.  In one case, $M\cpj M'$, we obtain a significantly stronger result, producing an explicit formula for the Jordan canonical form of the product of $M$ and $M'$  in terms of the Jordan canonical forms of $M$ and $M'$.  This allows the determination of the geometric multiplicities of the eigenvalues of the construction from the geometric multiplicities of the eigenvalues of $M$ and $M'$. We also show that such a determination is not possible for $M\cpi M'$ (see Example \ref{ex:geometric_mult_cartesian_prod_counterexample}).
In another case, $M\lexp M'$, we determine the geometric multiplicities of the eigenvalues of the construction from the geometric multiplicities of the eigenvalues of $M$ and $M'$ and the geometry of the eigenspaces of $M$ and $M'$.  %For many cases, we determine the spectrum of the construction of $M$ and $M'$ by using the construction of Jordan canonical forms of $M$ and $M'$ to obtain a triangular matrix that is similar to the construction of $M$ and $M'$.  In one case, $M\cpj M'$, we obtain a significantly stronger result, producing an explicit formula for the Jordan canonical form of the product of $M$ and $M'$  in terms of the Jordan canonical forms of $M$ and $M'$.  This allows the determination of the geometric multiplicities of the eigenvalues from the geometric multiplicities of the eigenvalues of $M$ and $M'$. We also show that such a determination can be problematic because in at least one other case, $M\cpi M'$, we show that such a determination is not possible (see Example \ref{ex:geometric_mult_cartesian_prod_counterexample}). In another case, $M\lexp M'$, with additional conditions on $M$ and $M'$ we determine the geometric multiplicities of the eigenvalues from the geometric multiplicities of the eigenvalues of $M$ and $M'$ and the relationship between the eigenvalues of $M$ and $M'$.  %These are startling results, 

\subsection{Useful lemmas}

The following result is used throughout the paper (there are many ways it could be proved).

\begin{lemma}
\label{lem_sylvester_1645_10nov}
Consider the block matrices $E=\begin{bmatrix}
A & C\\
O & B
\end{bmatrix}
$ 
and
$F=\begin{bmatrix}
A & O\\
O & B
\end{bmatrix}
$ 
where  $A\in\Cnn$,  $B\in\C^{n'\times n'}$, and suppose that $\spec(A)\cap\spec(B)=\emptyset$. Then $E$ and $F$ are similar.
\end{lemma}
\begin{proof}
The Sylvester equation $AX-XB=C$ has a unique solution $X \in \C^{n \times n'}$,  since $\spec(A)\cap\spec(B)=\emptyset$ \cite [Theorem 2.4.4.1]{HJ}. Then, $P^{-1}EP=F$ where  $P=\begin{bmatrix}
\I & -X\\
O & \I
\end{bmatrix}
$; observe that
$P^{-1}=\begin{bmatrix}
\I & X\\
O & \I
\end{bmatrix}
$.
\end{proof}
The next lemma is well known, and follows from standard facts about Kronecker products (see, for example,  \cite[Fact 11.4.16]{HLA-Ch11}).

\begin{lemma}\label{lem:kron-basis}  Let $\ba_1,\dots,\ba_k\in\R^n$ be linearly independent and let $\bb_1,\dots,\bb_{k'}\in\R^{n'}$ be linearly independent.  Then 
$\ba_i\otimes\bb_j$ for $i=1,\dots,k, j=1,\dots,k'$
are linearly independent in $\R^{nn'}$.
\end{lemma}
% \bpf Define $A=[\ba_1\cdots\ba_n]$ and $B=[\bb_1\cdots\bb_m]$.  By \cite[Fact 11.4.16]{HLA-Ch11} \[\rank(A\otimes B)=(\rank A)(\rank B)=nm.\]The columns of $A\otimes B$ are the vectors in $\{\ba_i\otimes\bb_j: i=1,\dots,n; j=1,\dots,m\}$.\epf
%
%

%%%%%%%%%%%%%%%%%%%%%%%%%%%%%%%%%%%%%%%%%%%%

\section{Cartesian products}\label{scartprod}

In this section we derive formulas for the  adjacency and distance  spectra of a Cartesian product of two digraphs  in terms of  the adjacency and distance spectra of the digraphs under certain conditions.  In the case of out-regular digraphs for adjacency matrices, or transmission regular digraphs for distance matrices, these results extend naturally to the (distance) Laplacian and (distance) signless Laplacian matrices.
These formulas show that Cartesian products provide a method for building infinite families of transmission regular digraphs with few distinct distance eigenvalues; this is discussed in Section \ref{sSRD}.  
The formulas (and the idea of constructing digraphs with few distance eigenvalues) parallel similar results for graphs.  However, the proofs of the eigenvalue formulas are quite different.  
 
% Specifically, let $\dig$ and $\dig'$ be digraphs of orders $n$ and $n'$.  Then as in the case with graphs,  $\A(\dig\cp\dig')=\A(\dig)\otimes \I_{n'}+\I_n\otimes\A(\dig')$ and $\D(\dig\cp\dig')=\D(\dig)\otimes \J_{n'}+\J_n\otimes\D(\dig')$.  
 %There are known formulas for the  adjacency and   distance  spectra of a Cartesian product of two graphs  in terms of  the   adjacency or distance spectra of the graphs.   
 Formulas analogous to the ones we derive for digraphs are known for graphs.  
 In the case of graphs, each of the matrices involved is real and symmetric,  so its eigenvalues are real and there is a basis of eigenvectors.  Furthermore, the distance matrix of a transmission regular graph $G$ of order $n$ commutes with $\J_n$, allowing simultaneous diagonalization of $\D(G)$ and $\J_n$.  Unfortunately, the eigenvalues of distance or adjacency matrices of digraphs  may be non-real  and there may not be a basis of eigenvectors.  Examples include the directed cycle $\vec C_n$, which has eigenvalues $1,\omega,\dots,\omega^{n-1}$ where $\omega=e^{(2\pi i)/n}$. A transmission regular digraph of diameter two that lacks a basis of eigenvectors is exhibited in the next example.
 
\begin{figure}[h!]
\begin{center}
\scalebox{.8}{\begin{tikzpicture}
\definecolor{cv0}{rgb}{0.0,0.0,0.0}
\definecolor{cfv0}{rgb}{1.0,1.0,1.0}
\definecolor{clv0}{rgb}{0.0,0.0,0.0}
\definecolor{cv1}{rgb}{0.0,0.0,0.0}
\definecolor{cfv1}{rgb}{1.0,1.0,1.0}
\definecolor{clv1}{rgb}{0.0,0.0,0.0}
\definecolor{cv2}{rgb}{0.0,0.0,0.0}
\definecolor{cfv2}{rgb}{1.0,1.0,1.0}
\definecolor{clv2}{rgb}{0.0,0.0,0.0}
\definecolor{cv3}{rgb}{0.0,0.0,0.0}
\definecolor{cfv3}{rgb}{1.0,1.0,1.0}
\definecolor{clv3}{rgb}{0.0,0.0,0.0}
\definecolor{cv0v1}{rgb}{0.0,0.0,0.0}
\definecolor{cv0v3}{rgb}{0.0,0.0,0.0}
\definecolor{cv1v0}{rgb}{0.0,0.0,0.0}
\definecolor{cv1v2}{rgb}{0.0,0.0,0.0}
\definecolor{cv2v0}{rgb}{0.0,0.0,0.0}
\definecolor{cv2v1}{rgb}{0.0,0.0,0.0}
\definecolor{cv3v0}{rgb}{0.0,0.0,0.0}
\definecolor{cv3v2}{rgb}{0.0,0.0,0.0}
\Vertex[style={minimum
size=1.0cm,draw=cv0,fill=cfv0,text=clv0,shape=circle},LabelOut=false,L=\hbox{$1$},x=2.1397cm,y=0.0cm]{v0}
\Vertex[style={minimum
size=1.0cm,draw=cv1,fill=cfv1,text=clv1,shape=circle},LabelOut=false,L=\hbox{$2$},x=5.0cm,y=1.208cm]{v1}
\Vertex[style={minimum
size=1.0cm,draw=cv2,fill=cfv2,text=clv2,shape=circle},LabelOut=false,L=\hbox{$3$},x=2.7507cm,y=5.0cm]{v2}
\Vertex[style={minimum
size=1.0cm,draw=cv3,fill=cfv3,text=clv3,shape=circle},LabelOut=false,L=\hbox{$4$},x=0.0cm,y=3.7133cm]{v3}
\Edge[lw=0.1cm,style={color=cv0v1,},](v0)(v1)
\Edge[lw=0.1cm,style={color=cv0v1,},](v0)(v3)
% \Edge[lw=0.1cm,style={post, bend right,color=cv1v0,},](v1)(v0)
\Edge[lw=0.1cm,style={color=cv0v1,},](v1)(v2)
\Edge[lw=0.1cm,style={post,color=cv2v0,},](v2)(v0)
% \Edge[lw=0.1cm,style={post, bend right,color=cv2v1,},](v2)(v1)
% \Edge[lw=0.1cm,style={post, bend right,color=cv3v0,},](v3)(v0)
\Edge[lw=0.1cm,style={post,,color=cv3v2,},](v3)(v2)
\end{tikzpicture}}
\caption{\label{fig:TRegNoEvec} A transmission regular digraph with diameter two and no basis of eigenvectors. Here and elsewhere, a bold line indicates both arcs are present. }\vspace{-5pt}
\end{center}
\end{figure}
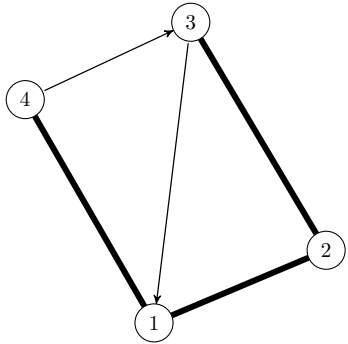
 \begin{example} {\rm Let $\dig$ be the digraph shown in Figure \ref{fig:TRegNoEvec}.  Then  $\D(\dig)=\mtx{0 & 1 & 2 & 1 \\
1 & 0 & 1 & 2 \\
1 & 1 & 0 & 2 \\
1 & 2 & 1 & 0}$, $\dspec(\dig)=\{4,  -1, -1, -2\}$, and every eigenvector for $-1$ is a multiple of $[4, -1, -1, -1]^T$.
%\[\mtx{1\\ 1\\ 1\\ 1}, \mtx{1\\ -1/4\\ -1/4\\ -1/4}, \mtx{1\\ -1\\ -1\\ 1}\] is a maximal set of independent eigenvectors.  
 }\end{example}

%In the present section we generalize this result to digraphs (Theorem \ref{thm:TRcartprod-dig_new_new}). In addition, we establish results about the Jordan canonical form and the eigenvectors of   the distance matrix of the Cartesian product of  two transmission regular digraphs (Theorems and  \ref{thm:evectors_cart_prod_digraphs}). 

%\begin{obs} \label{obs:dist-cart} Let $\dig$ and $\dig'$ be digraphs of orders $n$ and $n'$.  Then as in the case with graphs, $\D(\dig\cp\dig')=\D(\dig)\otimes J_{n'}+J_n\otimes\D(\dig')$.\end{obs}

%Rather than  establishing spectral results just for the adjacency or distance matrices of Cartesian products, we develop a general theory of the spectra of matrices constructed in a specified form  as a sum of Kronecker products of arbitrary matrices with the identity or with the all ones matrix: 
%Let $M\in\Cnn$ and $M'\in\C^{n'\times n'}$.  Define the matrix \[M\cpi M'=M\otimes \I_{n'}+\I_n\otimes M'\in\mathbb{R}^{(nn')\times (nn')}.\] 
%Define the matrix 
%\[M\cpj M'=M\otimes \J_{n'}+\J_n\otimes M'\in\mathbb{R}^{(nn')\times (nn')}.\] 

If the matrices $M$ and $M'$ are real and symmetric, then formulas for the spectra of $M\cpi M'$ and $M\cpj M'$ in terms of those of $M$ and $M'$ are well known. The formula for $\spec(M\cpi M')$ is also known without any other assumptions.
%Here we use the Jordan canonical form to prove analogous results about the spectra for digraphs and then also determine information about eigenvectors; the Jordan canonical form of $M$ is denoted $\JCF_M$.  The results apply to adjacency and distance matrices of digraphs.  
%and have spectra $\lam_1,\dots,\lam_n$ and $\lam'_1,\dots,\lam'_{n'}$, respectively, it is well known that $\spec(M\cpi M')=\{\lam_i+\lam_j : i=1,\dots,n, j=1,\dots,n'\}$ (see, for example, \cite[pages 206--207]{GR01}). We use a different method of proof to extend this to not necessarily symmetric complex matrices.

\begin{remark}{\rm   Let $M\in\Cnn$ with $\spec(M)=\{\lam_1,\dots,\lam_n\}$ and  $M'\in\C^{n'\times n'}$ with $\spec(M')=\{\lam'_1,\dots,\lam'_{n'}\}$.
Then
$\spec(M\cpi M')=\{\lam_i+\lam_j' : i=1,\dots,n, \ j=1,\dots,n'\}$ \cite[Theorem 4.4.5]{HJ2}.
This implies the (known) formula for   the adjacency spectra of cartesian products of any digraphs: 
   Let $\dig$ and $\dig'$ be digraphs of orders $n$ and $n'$, respectively, with 
$\spec_{\A}(\dig)=\{\alpha_1,\alpha_2,\dots,\alpha_n\}$ and $\spec_{\A}(\dig')=\{\alpha_1',\alpha'_2,\dots,\alpha'_{n'}\}$ \cite[Theorem 3]{EH80}.  
Then
$
\spec_{\A}(\dig{\cp}\dig')=\left\{\alpha_i+\alpha_j' : i=1,\dots,n, \  j=1,\dots,n' \right\}.
$
}\end{remark}

As the next example shows, the geometric multiplicity of the eigenvalues of $M\cpi M'$ is not entirely determined from the eigenvalues of $M$ and $M'$ and  their geometric multiplicities.

\begin{example}
\label{ex:geometric_mult_cartesian_prod_counterexample}
{\rm Let 
\[
M=
\begin{bmatrix}
0& 0 & 0\\
0 & 0 & 1\\
0 & 0 & 0
\end{bmatrix},
\hspace{.4cm}
M'_1=\begin{bmatrix}
0 & 1 & 0 & 0\\
0 & 0 & 0 & 0\\
0 & 0 & 0 & 1\\
0 & 0 & 0 & 0
\end{bmatrix},
\hspace{.4cm}
M'_2=\begin{bmatrix}
0 & 1 & 0 & 0\\
0 & 0 & 1 & 0\\
0 & 0 & 0 & 0\\
0 & 0 & 0 & 0
\end{bmatrix}.
\]
We observe that the eigenvalue $0$ has geometric multiplicity $2$ (and algebraic multiplicity 4) for both $M'_1$ and $M'_2$. Nevertheless, one can check that $\rank(M\cpi M'_1)=6$ while $\rank(M\cpi M'_2)=7$, so that the geometric multiplicity of $0$ for $M\cpi M'_1$ and for $M\cpi M'_2$ differs.
}\end{example}

 The formula for $\spec(M\cpi M')$ can be proved by using the Jordan canonical form (as we do in other theorems) %and then also determine information about eigenvectors; the Jordan canonical form of $M$ is denoted $\JCF_M$.  The results apply to adjacency and distance matrices of digraphs.  
The geometric multiplicity of the eigenvalues of $M\cpi M'$ is fully determined from the Jordan canonical forms of $M$ and $M'$; these, in turn, are fully determined from their  Weyr characteristics (see for example \cite[\S $3.1$]{HJ1}). We conclude that, in addition to the geometric multiplicities, other elements of the Weyr characteristics of $M$ and $M'$ determine the geometric multiplicity of the eigenvalues of $M\cpi M'$.

%{\red If we have any eigenvector result for $M\cpi M'$ it goes here. }\medskip

Next we turn our attention to $M\cpj M'$.  

%=\diag(1,0,\dots,0)

%Let $M$ be a nonnegative matrix and denote its spectral radius by $\rho$.  If $M\bone=\alpha\bone$ for some $\alpha\in \R$, then $\alpha=\rho$. Now suppose $M$ is irreducible.  Recall from the Perron-Frobenius Theorem that $\rho$ is a simple eigenvalue of  $M$.
 
\begin{proposition}\label{prop:MJ-JCF}
Suppose $M\in\Rnn$ is a nonnegative matrix that satisfies $M\bone_n=\rho\bone_n$.  Then  $\rho$ is the spectral radius of $M$ and  there exists an invertible matrix $C\in\Cnn$ such that 
\[C^{-1}\J_n C=\mtx{n & \boldzero^T \\ \boldzero & O} \mbox{ and } C^{-1}MC=\mtx{\rho & \bx^T \\ \boldzero & R}\]
for some   Jordan matrix $R$ and  $\bx\in\R^{n-1}$.  If in addition $M$ is irreducible, then  $\JCF_M=\mtx{\rho & \boldzero^T \\ \boldzero & R}$.
\end{proposition}

\bpf Since $ M$ is a nonnegative matrix that satisfies $M\bone_n=\rho\bone_n$, its spectral radius is $\rho$.  Choose a basis of (real) eigenvectors $\bc_1=\bone, \bc_2,\dots,\bc_n$ for $\J_n$ and define $C_1=\mtx{\bc_1 & \bc_2 &\dots &\bc_n}$.  Then 
$C_1^{-1}\J_n C_1=\mtx{n & \boldzero^T \\ \boldzero & O}$ and $ C_1^{-1}MC_1=\mtx{\rho & \by^T \\ \boldzero & B}$.  Choose $C_2\in\C^{(n-1)\x(n-1)}$ such that $C_2^{-1}BC_2=\JCF_B$.  Then 
$C^{-1}MC=\mtx{\rho & \bx^T \\ \boldzero & R}$ with $C=C_1([1]\oplus C_2)$ and $R=\JCF_B$.
If $ M$ is  irreducible, then  $\rho$ is a simple eigenvalue and $\JCF_M$ has the required form. \epf

Observe that any Jordan matrix $R$ can be expressed as $R=D+N$ where $D$ is a diagonal matrix and $N$ is nilpotent. Then for any $c\in\R$, $\JCF_{cR}=cD+N$.

\begin{theorem}\label{thm:M-JCF-CartProd}
Suppose $M\in\Rnn$, $M'\in\R^{n'\times n'}$ are irreducible nonnegative matrices that satisfy $M\bone_n=\rho\bone_n$ and  $M'\bone_{n'}=\rho'\bone_{n'}$. Let $\JCF_M=\begin{bmatrix}
\rho & \boldzero^T \\ 
\boldzero & D+N
\end{bmatrix}$ and $\JCF_{M'}=\begin{bmatrix}
\rho' & \boldzero^T \\ 
\boldzero & D'+N'
\end{bmatrix}$, where $D$ and $D'$ are diagonal and $N$ and $N'$ are nilpotent. Then
\[\JCF_{M\cpj M'}=
\begin{bmatrix}
n\rho'+n'\rho & \boldzero^T & \boldzero^T & \boldzero^T \\ 
\boldzero & nD'+N' & O & O \\ 
\boldzero & O & n'D+N & O \\ 
\boldzero & O & O & O
\end{bmatrix} .
\]

\end{theorem}
\bpf  Let $R=D+N$ and $R'=D'+N'$.  Use Proposition \ref{prop:MJ-JCF} to choose $C$ and $C'$ such that $C^{-1}\J_n C=\mtx{n & \boldzero^T \\ \boldzero & O}=\diag(n,0,\dots,0)$, $ C^{-1}MC=\mtx{\rho & \bx^T \\ \boldzero & R}$,  $C'^{-1}\J_{n'} C'=\mtx{n' & \boldzero^T \\ \boldzero & O}=\diag(n',0,\dots,0)$, and $ C'^{-1}M'C'=\mtx{\rho' & \bx'^T \\ \boldzero & R'}$. % where $JCF(M)=\mtx{\rho & \boldzero^T \\ \bzero & R}$ and $JCF(M')=\mtx{\rho' & \bzero^T \\ \bzero & R'}$.
%Note that $R$ and  $R'$ are in Jordan canonical form. %, associated with eigenvalues $\lam'_2,\dots,\lam'_{n'}$, all of which differ from $\rho'$. 
%s the part of $JCF(M')$ associated with eigenvalues $\lam'_2,\dots,\lam'_{n'}$, all of which differ from $\rho'$. 
Then \\
$(C^{-1}\otimes C'^{-1})(M\cpj M')(C\otimes C')=$\\
$\mtx{\rho & \bx^T \\ \boldzero & R}\otimes \diag(n',0,\dots,0)+\diag(n,0,\dots,0)\otimes \mtx{\rho' & \bx'^T \\ \boldzero & R'}=$ \\
\renewcommand{\arraystretch}{1.3}
\[{\scriptsize \lb \begin{array}{cc|cc|cc|c|cc} 
\rho n' & \boldzero^T & x_1n' & \boldzero^T& x_2n' & \boldzero^T & \cdots & x_{n-1}n' &\boldzero^T \\ 
\boldzero & O & \boldzero & O & \boldzero & O& \cdots & \boldzero & O    \\
\hline
 0 & \boldzero^T& r_{11}n' & \boldzero^T & r_{12}n' & \boldzero^T &  \cdots & 0 &\boldzero^T \\ 
\boldzero & O & \boldzero & O & \boldzero & O & \cdots & \boldzero & O  \\ 
\hline
 0 & \boldzero^T  & 0 & \boldzero^T & r_{22}n' & \boldzero^T&  \cdots & 0 &\boldzero^T \\ 
\boldzero & O & \boldzero & O & \boldzero & O & \cdots & \boldzero & O  \\ 
\hline
\vdots & \vdots &\vdots &\vdots &\vdots &\vdots & \ddots &\vdots &\vdots\\
\hline
%
% 0 & \boldzero^T  & 0 & \boldzero^T & 0 & \boldzero^T&  \cdots & r_{n-1,n}n' &\boldzero^T \\ 
%\boldzero & O & \boldzero & O & \boldzero & O & \cdots & \boldzero & O  \\ 
%
%\hline
 0 & \boldzero^T  & 0 & \boldzero^T & 0 & \boldzero^T&  \cdots & r_{nn}n' &\boldzero^T \\ 
\boldzero & O & \boldzero & O & \boldzero & O & \cdots & \boldzero & O  \\ 
\end{array}\rb +  
%% ++
\lb \begin{array}{cc|cc|cc|c|cc} 
n \rho' & n \bx'^T & 0 & \boldzero^T& 0 & \boldzero^T & \cdots & 0 &\boldzero^T \\ 
\boldzero & n R' & \boldzero & O & \boldzero & O& \cdots & \boldzero & O    \\
\hline
 0 & \boldzero^T& 0 & \boldzero^T& 0 & \boldzero^T &  \cdots & 0 &\boldzero^T \\ 
\boldzero & O & \boldzero & O & \boldzero & O & \cdots & \boldzero & O  \\ 
\hline
 0 & \boldzero^T  & 0 & \boldzero^T & 0 &\boldzero^T &  \cdots & 0 &\boldzero^T \\ 
\boldzero & O & \boldzero & O & \boldzero & O & \cdots & \boldzero & O  \\ 
\hline
\vdots & \vdots &\vdots &\vdots &\vdots &\vdots & \ddots &\vdots &\vdots\\
\hline
%
% 0 & \boldzero^T  & 0 & \boldzero^T & 0 & \boldzero^T&  \cdots & 0 &\boldzero^T \\ 
%\boldzero & O & \boldzero & O & \boldzero & O & \cdots & \boldzero & O  \\ 
%
%\hline
 0 & \boldzero^T  & 0 & \boldzero^T & 0 & \boldzero^T&  \cdots & 0 &\boldzero^T \\ 
\boldzero & O & \boldzero & O & \boldzero & O & \cdots & \boldzero & O  
\end{array}\rb=
}\]
%%==
\beq\label{eq:CP1} \lb \begin{array}{cc|cc|cc|c|cc} 
\rho n' +n\rho' & n\bx'^T &  x_1n' & \boldzero^T& x_2n' & \boldzero^T & \cdots & x_{n-1}n' &\boldzero^T  \\ 
\boldzero & nR' & \boldzero & O & \boldzero & O& \cdots & \boldzero & O    \\
\hline
 0 & \boldzero^T& r_{11}n' & \boldzero^T & r_{12}n' & \boldzero^T &  \cdots & 0 &\boldzero^T \\ 
\boldzero & O & \boldzero & O & \boldzero & O & \cdots & \boldzero & O  \\ 
\hline
 0 & \boldzero^T  & 0 & \boldzero^T & r_{22}n' & \boldzero^T&  \cdots & 0 &\boldzero^T \\ 
\boldzero & O & \boldzero & O & \boldzero & O & \cdots & \boldzero & O  \\ 
\hline
\vdots & \vdots &\vdots &\vdots &\vdots &\vdots & \ddots &\vdots &\vdots\\
\hline
%
 %0 & \boldzero^T  & 0 & \boldzero^T & 0 & \boldzero^T&  \cdots & r_{n-1,n}n' &\boldzero^T \\ 
%\boldzero & O & \boldzero & O & \boldzero & O & \cdots & \boldzero & O  \\ 
%
%\hline
 0 & \boldzero^T  & 0 & \boldzero^T & 0 & \boldzero^T&  \cdots & r_{nn}n' &\boldzero^T \\ 
\boldzero & O & \boldzero & O & \boldzero & O & \cdots & \boldzero & O  \\ 
\end{array}\rb\!\!.
\eeq
\renewcommand{\arraystretch}{1}
The matrix in \eqref{eq:CP1} is permutation similar to
\beq\label{eq:CP2} \lb \begin{array}{cccccc} 
\rho n' +n\rho' & n\bx'^T &  n'\bx^T & \boldzero^T&  \cdots & \boldzero^T  \\ 
\boldzero & nR' &  O  & O& \cdots &  O    \\
 \boldzero& O & n'R & O & \cdots & O \\ 
\boldzero & O  & O &  O & \cdots &  O  \\ 
\vdots & \vdots &\vdots &\vdots  & \ddots &\vdots\\
\boldzero & O  & O &  O & \cdots &  O 
\end{array}\rb\!\!.
\eeq
Since $\rho n' +n\rho'$ is not an eigenvalue  of  $nR'$ or $n'R$, Lemma \ref{lem_sylvester_1645_10nov} implies that  the Jordan canonical form of the matrix in \eqref{eq:CP2} is \[\begin{bmatrix}
n\rho'+n'\rho & \boldzero^T & \boldzero^T & \boldzero^T \\ 
\boldzero & \JCF_{nR'} & O & O \\ 
\boldzero & O & \JCF_{n'R} & O \\ 
\boldzero & O & O & O
\end{bmatrix}=\begin{bmatrix}
n\rho'+n'\rho & \boldzero^T & \boldzero^T & \boldzero^T \\ 
\boldzero & nD'+N' & O & O \\ 
\boldzero & O & n'D+N & O \\ 
\boldzero & O & O & O
\end{bmatrix}\!\!.\]
\epf

\begin{corollary}\label{cor:spectrum_cart_prod_general_matrices} Suppose $M\in\Rnn$, $M'\in\R^{n'\times n'}$ are irreducible nonnegative matrices that satisfy $M\bone_n=\rho\bone_n$ and $M'\bone_{n'}=\rho'\bone_{n'}$. Let $\spec(M)=\{\rho,\lambda_2,\dots,\lambda_n\}$ and $\spec(M')=\{\rho',\lambda'_2,\dots,\lambda'_{n'}\}$. Then
\[
\spec(M\cpj M')=\{n\rho'+n'\rho,n'\lambda_2,\dots,n'\lambda_n,n\lambda'_2,\dots,n\lambda'_{n'},0^{(n-1)(n'-1)}\}.
\]
\end{corollary}

Considering $M$ and $M'$ in Corollary \ref{cor:spectrum_cart_prod_general_matrices} to be the distance matrices of two transmission regular digraphs we immediately obtain the next result.

\begin{theorem}\label{thm:TRcartprod-dig_new_new} Let $\dig$ and $\dig'$ be transmission regular digraphs of orders $n$ and $n'$ with transmissions $t$ and $t'$, and let $\spec_{\D}(\dig)=(t,\dev_2,\dots,\dev_n)$, $\spec_{\D}(\dig')=(t',\dev'_2,\dots,\dev'_{n'})$.  
Then
\[\spec_{\D}(\dig{\cp}\dig')=\{nt'+n't, n'\dev_2,\dots,n'\dev_n,n\dev'_2,\dots,n\dev'_{n'},0^{(n-1)(n'-1)}\}.\]
\end{theorem}

The formula for the distance spectrum of a Cartesian product of graphs (analogous to that in Theorem \ref{thm:TRcartprod-dig_new_new}) was originally proved by Indulal for distance regular graphs \cite[Theorem 2.1]{I2009}, and  it was noted in \cite{AP15} that the proof applies to transmission regular graphs. The proof used the facts that the distance matrix of a transmission regular graph commutes with $\J$ and every real symmetric matrix has a basis of eigenvectors.

%\begin{theorem}\label{thm:TRcartprod}{\rm \cite[Theorem 2.1]{I2009}}  Let $G$ and $G'$ be transmission regular graphs of orders $n\ge 2$ and $n'\ge 2$, respectively, with $\dspec(G)=\{t(G)\}\cup\{ \dev_k:k=2,\dots,n\}$ and $\dspec(G')=\{t(G')\}\cup\{ \dev'_k:k=2,\dots,n'\}$.  Then, $G \cp G'$ is transmission regular and   \[\dspec(G \cp G')=\{ n' {t(G)}+n {t(G')}\}\cup\{ n' \dev_2,\dots,n'\dev_n\}\cup\{ n \dev'_2,\dots,n\dev'_{n'}\} \cup \{0^{(n-1)(n'-1)}\}.\] \end{theorem}

Having found the spectrum of $\D(\dig\cp\dig')$, we now focus on describing its eigenvectors.

\begin{theorem}
\label{prop:evectors_cart_prod_general_matrices}
Let $M\in\mathbb{R}^{n\times n}$ and $M'\in\mathbb{R}^{n'\times n'}$ be irreducible nonnegative matrices, and suppose that $M\bone_n=\rho\bone_n$, $M'\bone_{n'}=\rho'\bone_{n'}$ for some $\rho,\rho'\geq 0$. Let $\{\bv_2, \dots, \bv_k\}$ be a linearly independent set of eigenvectors of $M$ with $M \bv_i = \lambda_i \bv_i, \ \lambda_i \in \spec(M)$, and let $\{\bv_2', \dots, \bv_{k'}'\}$ be a linearly independent set of eigenvectors of $M'$ with $M' \bv_{j}' = \lambda_j' \bv_j', \ \lambda_j' \in \spec(M')$. Then 
\ben[$(1)$]
\item \label{cartevec_1}  $\bone_n\otimes\bone_{n'}$ is an eigenvector of ${M}{\cpj}{M}'$ corresponding to the spectral radius,  $n\rho'+n'\rho$.
\item \label{cartevec_2} For $i=2, \dots, k$, \ 
$\bv_i \otimes \bone_{n'} + \gamma_i\bone_n \otimes \bone_{n'}$, where $\gamma_i =\frac{\bv_i^T\bone_n \rho'}{n'\lambda_i-n'\rho - n\rho'}$, is an eigenvector of ${M}{\cpj}{M}'$ corresponding to the eigenvalue  $n'\lambda_i$.
\item \label{cartevec_3} For $j=2, \dots, k'$, \
$\bone_n \otimes \bv_j' +\gamma'_j\bone_n \otimes \bone_{n'}$, where $\gamma_j' =\frac{\bv_j'^T\bone_{n'} \rho}{n\lambda'_j-n\rho' - n'\rho}$, is an eigenvector of $M\cpj M'$ corresponding to  the eigenvalue  $n\lambda'_j$.
\item \label{cartnullvec} Let $\{\bz_1, \dots, \bz_{n-1}\}$, respectively $\{\bz'_1, \dots, \bz'_{n'-1}\}$, be a linearly independent set of null vectors of $\J_n$, respectively $\J_{n'}$. Then, for $i=1, \dots, n-1, \ j=1, \dots, n'-1$,  $\bz_i \otimes \bz_{j}'$ is a null vector of ${M}{\cpj}{M}'$.
\een
Furthermore, the set of eigenvectors of ${M}{\cpj}{M}'$ described in \eqref{cartevec_1}--\eqref{cartnullvec} is linearly independent.  If  $\{\bv_1=\bone_n,\bv_2, \dots,$ $ \bv_n\}$ and $\{\bv_1'=\bone_{n'},\bv_2', \dots, \bv_{n'}'\}$ are bases of eigenvectors for $M$ and $M'$, then  the set of eigenvectors of ${M}{\cpj}{M}'$ described in \eqref{cartevec_1}--\eqref{cartnullvec} is a basis of eigenvectors.
\end{theorem}
\begin{proof} $\null$
\ben[(1)]
\item  ${M}{\cpj}{M}'(\bone_n\otimes\bone_{n'}) = [{M} \otimes \J_{n'}+\J_n\otimes{M}'](\bone_n\otimes\bone_{n'}) = ({\rho}\bone_n \otimes n' \bone_{n'}) + (n \bone_n \otimes {\rho}' \bone_{n'}) =  {\rho}n' (\bone_n\otimes\bone_{n'}) + n{\rho}'(\bone_n\otimes\bone_{n'}) = (n{\rho}' + n'{\rho}) (\bone_n\otimes\bone_{n'})$.

\item  For simplicity,  let $\bv = \bv_i, \lambda = \lambda_i$, and $\gamma = \gamma_i$.  As $|\lambda| \leq {\rho}$, $\gamma$ is well-defined and satisfies   $(\bv^T\bone_n){\rho}' + \gamma {\rho} n' + \gamma n {\rho}' - n'\lambda \gamma = 0$. Moreover, ${M}{\cpj}{M}' (\bv \otimes \bone_{n'} + \gamma \bone_n \otimes \bone_{n'}) =  [{M} \otimes \J_{n'} + \J_n\otimes{M}'](\bv \otimes \bone_{n'} + \gamma \bone_n \otimes \bone_{n'})  =  ({M} \otimes \J_{n'})(\bv \otimes \bone_{n'}) +  (\J_n\otimes{M}')(\bv \otimes \bone_{n'})  + ({M} \otimes \J_{n'})(\gamma \bone_n \otimes \bone_{n'})  +  (\J_n\otimes{M}')(\gamma \bone_n \otimes \bone_{n'}) =  
\lambda \bv \otimes n'\bone_{n'} + (\bv^T\bone_n) \bone_n \otimes {\rho}'\bone_{n'} + \gamma {\rho} \bone_n \otimes n' \bone_{n'} + \gamma n \bone_n \otimes {\rho}'\bone_{n'} =  
n'\lambda (\bv \otimes \bone_{n'}  + \gamma \bone_n \otimes \bone_{n'}) + ((\bv^T\bone_n){\rho}' + \gamma {\rho} n' + \gamma n {\rho}' - n'\lambda \gamma)( \bone_n \otimes \bone_{n'})=n'\lambda (\bv \otimes \bone_{n'}  + \gamma \bone_n \otimes \bone_{n'})$.

\item The proof is analogous to that of \eqref{cartevec_2}. %For simplicity,  let $\bv' = \bv'_j, \lambda' = \lambda'_j$, and $\gamma' = \gamma'_j$. As $|\lambda'| \leq {\rho}'$, $\gamma'$ is well-defined and satisfies $\bv'^T\bone_{n'} {\rho} +\gamma'{\rho} n' +\gamma'n {\rho}' - n \lambda'\gamma'= 0$. Moreover, ${M}{\cpj}{M}'(\bone_n \otimes \bv' +\gamma'\bone_n \otimes \bone_{n'}) =  [{M} \otimes \J_{n'} + \J_n\otimes{M}'](\bone_n \otimes \bv' +\gamma'\bone_n \otimes \bone_{n'})  =  ({M} \otimes \J_{n'})(\bone_n \otimes \bv') +  (\J_n\otimes{M}')(\bone_n \otimes \bv')   + ({M} \otimes \J_{n'})(\gamma' \bone_n \otimes \bone_{n'})    +  (\J_n\otimes{M}')(\gamma' \bone_n \otimes \bone_{n'})   = ({\rho} \bone_n \otimes (\bv'^T\bone_{n'}) \bone_{n'}) + (n \bone_n \otimes \lambda' \bv') +  (\gamma' {\rho} \bone_n \otimes n' \bone_{n'}) + (\gamma' n \bone_n \otimes {\rho}' \bone_{n'}) =  n\lambda' (\bone_n \otimes \bv' +\gamma'\bone_n \otimes \bone_{n'}) + (\bv'^T\bone_{n'} {\rho} +\gamma'{\rho} n' +\gamma'n {\rho}' - n \lambda' \gamma')(\bone_n \otimes \bone_{n'})=n\lambda' (\bone_n \otimes \bv' +\gamma'\bone_n \otimes \bone_{n'})$.

\item  ${M}{\cpj}{M}'(\bz_i \otimes\bz'_j)  =  
[{M} \otimes \J_{n'} + \J_n\otimes{M}'](\bz_i \otimes \bz'_j) = 
(M \otimes \J_{n'})(\bz_i \otimes \bz'_j) + (\J_n \otimes M')(\bz_i \otimes \bz'_j) = 
M\bz_i \otimes \J_{n'} \bz'_j  + \J_n \bz_i \otimes M' \bz'_j = 
M\bz_i \otimes \boldzero +  \boldzero \otimes M' \bz'_j  = \boldzero$.\vspace{-5pt}
 \een

Note that  $(\bz_i \otimes \bz_{j}')^T(\bv\otimes\bone_{n'}) = (\bz_i^T \otimes \bz_{j}'^T)(\bv \otimes\bone_{n'}) = \bz_i^T \bv \otimes {\bz_j'}^T \bone_{n'} = 0$  for any vector $\bv$, and similarly $(\bz_i \otimes \bz_{j}')^T(\bone_n\otimes\bv') = 0$ for any vector $\bv'$.  Thus, the null  vectors $\bz_i \otimes \bz_{j}'$  are orthogonal to the eigenvectors in \eqref{cartevec_1}--\eqref{cartevec_3}.  Moreover, the eigenvectors in \eqref{cartevec_1}--\eqref{cartevec_3} are linearly independent  by Lemma \ref{lem:kron-basis},  hence the eigenvectors of ${M}{\cpj}{M}'$ in \eqref{cartevec_1}--\eqref{cartnullvec} are linearly independent. The statement regarding being a basis follows from the dimension.
\end{proof}
Next we apply Theorem \ref{prop:evectors_cart_prod_general_matrices}  to provide a description of the eigenvectors of the Cartesian product of two transmission regular digraphs. 
\begin{theorem}\label{thm:evectors_cart_prod_digraphs} Let $\dig$ and $\dig'$ be transmission regular digraphs of orders $n$ and $n'$ with transmissions $t$ and $t'$. Let $\{\bv_1=\bone_n, \dots, \bv_k\}$ be a linearly independent set of eigenvectors of $\D(\dig)$ with $\bv_i $ an eigenvector corresponding to $\dev_i \in \dspec(\dig)$, and let $\{\bv_1'=\bone_{n'}, \dots, \bv_{k'}'\}$ be a linearly independent set of eigenvectors of $\D(\dig')$ with $\bv_j' $ an eigenvector corresponding to $\dev_j' \in \dspec(\dig')$. Then

\ben
\item \label{Dcartevec_1}  $\bone_n\otimes\bone_{n'}$ is an eigenvector of $\D(\dig\cp\dig')$  corresponding to the spectral radius,  $nt'+n't$. 
\item \label{Dcartevec_2} For $i=2, \dots, k$, \ 
$\bv_i \otimes \bone_{n'} + \gamma_i\bone_n \otimes \bone_{n'}$, where $\gamma_i =\frac{\bv_i^T\bone_n t'}{n'\dev_i-n't - nt'}$, is an eigenvector of $\D(\dig\cp\dig')$  corresponding to the eigenvalue  $n'\dev_i$.
\item \label{Dcartevec_3} For $j=2, \dots, k'$, \
$\bone_n \otimes \bv_j' +\gamma'_j\bone_n \otimes \bone_{n'}$, where $\gamma_j' =\frac{\bv_j'^T\bone_{n'} t}{n\dev'_j-nt' - n't}$, is an eigenvector of $\D(\dig\cp\dig')$ corresponding to the eigenvalue  $n\dev'_j$.
\item \label{Dcartnullvec} Let $\{\bz_1, \dots, \bz_{n-1}\}$, respectively $\{\bz'_1, \dots, \bz_{n'-1}\}$, be a linearly independent set of null vectors of $\J_n$, respectively $\J_{n'}$. Then, for $i=1, \dots, n-1, \ j=1, \dots, n'-1$,  $\bz_i \otimes \bz_{j}'$ is a null vector of $\D(\dig\cp\dig')$ 
\een
Furthermore, the set of eigenvectors of $\D(\dig\cp\dig')$ described in \eqref{Dcartevec_1}--\eqref{Dcartnullvec} is linearly independent. If  $\{\bv_1=\bone_n,\bv_2, \dots, \bv_n\}$ and $\{\bv_1'=\bone_{n'},\bv_2', \dots, \bv_{n'}'\}$ are bases of eigenvectors for $\D(\dig)$ and $\D(\dig')$, then  the set of eigenvectors of $\D(\dig\cp\dig')$ described in \eqref{Dcartevec_1}--\eqref{Dcartnullvec} is a basis of eigenvectors.
%
%\item If $\dev\ne t$ is an eigenvalue of $\D(\dig)$ with eigenvector $\bv$, then $n'\dev$ is an eigenvalue of $\D(\dig\cp\dig')$ with eigenvector 
%$\bv \otimes \bone_{n'} + \gamma \bone_n \otimes \bone_{n'}$ where $\gamma =\frac{\bv^T\bone_n t'}{n'\dev-n't - nt'}$.
%\item If $\dev' \ne t'$ is an eigenvalue of $\D(\dig')$ with eigenvector $\bv'$, then $n\dev'$ is an eigenvalue of $\D(\dig\cp\dig')$ with eigenvector 
%$\bone_n \otimes \bv' +\gamma'\bone_n \otimes \bone_{n'}$ where $\gamma' =\frac{\bv'^T\bone_{n'} t}{n\dev'-nt' - n't}$.
%
%\een
\end{theorem}
\begin{remark}{\rm 
If $\dig$ and $\dig'$ are symmetric digraphs (which is equivalent to considering them as undirected graphs), then their distance matrices are symmetric. As a consequence, $\gamma_i$ and $\gamma'_j$ are always zero in Theorem \ref{thm:evectors_cart_prod_digraphs}, which yields the simpler expression for the eigenvectors of $\D(\dig\cp\dig')$ used in \cite{I2009}.
}\end{remark}

 %Suppose $\dig$ is a transmission regular digraph of order $n$. Then $\DL(\dig)=t(\dig)I-\D(\dig)$  %The spectral radius of the distance matrix  $\D(\dig)$ is the transmission  $t=t(\dig)$ and the trace of the distance Laplacian is $\tr(\DL(\dig))=nt$.   Observe that  $\dLspec(\dig)=\{\dlev_1=0,\dlev_2,\dots,\dlev_n\}$ where $\dlev_k=t- \dev_k$ for $k=1,\dots,n$ with the same eigenvector.  
%  are $\D$-cospectral if and only if they are $\DL$-cospectral.  %Since  $\D(G)$ is nonnegative, its spectral radius is a simple eigenvalue.  Therefore,  $\rho(\D(G))=t(G)>\dev_i$ for $i>1$ and $\dlev_i(G)=t(G)-\dev_i>0$ for $i=2,\dots,n$.

Next we consider the distance Laplacian and the distance signless Laplacian matrices of a Cartesian product of transmission regular digraphs.
\begin{proposition} Let $\dig$ and $\dig'$ be transmission regular digraphs of orders $n\geq 2$ and $n'\geq 2$ with transmissions $t$ and $t'$ respectively, and let  $\dLspec(\dig)=\{{0, \dlev_2},\dots,\dlev_n\}$ and $\dLspec(\dig')=\{{0, {\dlev}'_2},\dots,{\dlev}'_{n'}\}$.  Then $t(\dig\cp \dig')=nt'+n't$ and   
\[\begin{aligned} \dLspec(\dig \cp \dig')&=\{ 0\}\cup\{ n t'+n' \dlev_2,\dots,n t'+n' \dlev_n\}\cup\{ n' t+n {\dlev}'_2,\dots,n' t+n {\dlev}'{n'}\}\\
& \null\hspace{5mm} \cup\ \{(nt'+n't)^{((n-1)(n'-1))}\}.\end{aligned}\]
\end{proposition} 
\bpf  Since $\DL(\dig)=t\I_n-\D(\dig)$,  $\dspec(\dig)=\{t,t-\dlev_2,\dots,t-\dlev_n\}$ and $\dspec(\dig')=\{t',t'-{\dlev}'_2,\dots,$ $t'-{\dlev}'_{n'}\}$.  Then by Theorem \ref{thm:TRcartprod-dig_new_new}, 
\[\begin{aligned} \dspec(\dig \cp \dig')&=\{nt'+ n't\}\cup\{ n' (t-\dlev_2),\dots,n'(t-\dlev_n)\}\\
& \null\hspace{5mm}\cup\ \{ n (t'-{\dlev}'_2),\dots,n(t'-{\dlev}'_{n'})\} \cup \{0^{((n-1)(n'-1))}\}.\end{aligned}\]
For transmission regular digraphs, the distance spectral radius is the transmission, so  $\trans(\dig \cp \dig')= nt'+n't$, and the formula for $\dLspec(\dig \cp \dig')$ follows from that for $\dspec(\dig \cp \dig')$.
\epf

\begin{proposition} Let $\dig$ and $\dig'$ be transmission regular digraphs of orders $n\geq 2$ and $n'\geq 2$ with transmissions $t$ and $t'$ respectively, and let $\dQspec(\dig)=\{2t,{\dQev_2},\dots,\dQev_{n}\}$ and $\dQspec(\dig')=\{2t',{{\dQev}'_2},\dots,{\dQev}'_{n'}\}$.  Then 
%$t(\dig\cp \dig')=n't+nt'$ and   
\[\begin{aligned} \dQspec(\dig \cp \dig')&=\{ 2nt'+2n't\}\cup\{ n t'+n' \dQev_2,\dots,n t'+n' \dQev_{n}\}\cup\{ n' t+n {\dQev}'_2,\dots,n' t+n {\dQev}'{n'}\}\\
& \null\hspace{5mm} \cup\ \{(nt'+n't)^{((n-1)(n'-1))}\}.\end{aligned}\]
\end{proposition} 

\bpf  Since $\DQ(\dig)=t\I_n+\D(\dig)$,  $\dspec(\dig)=\{t,\dQev_2-t,\dots,\dQev_{n}-t\}$ and $\dspec(\dig')=\{t',{\dQev}'_2-t',\dots,$ ${\dQev}'_{n'}-t'\}$.  Then by Theorem \ref{thm:TRcartprod-dig_new_new}, \vspace{-5pt}
\[\begin{aligned} \dspec(\dig \cp \dig')&=\{ nt'+n't\}\cup\{ n' (\dQev_2-t),\dots,n'(\dQev_{n}-t))\}\\ \vspace{-3pt}
& \null\hspace{5mm} \cup\ \{ n ({\dQev}'_2-t'),\dots,n({\dQev}'_{n'}-t')\} \cup \{0^{((n-1)(n'-1))}\}.\end{aligned}\]
The formula for $\dQspec(\dig \cp \dig')$ follows from that for $\dspec(\dig \cp \dig')$.
\epf

%%%%%%%%%%%%%%%%%%%%%%%%%%%%%%%%%%%%%%%%%%%%

\section{Lexicographic products}\label{slexprod}

Motivated by the results in \cite{I2009},  we investigate the spectra of lexicographic products of digraphs. 

Recall that for graphs  $G$ and $G'$  of orders $n$ and $n'$ the {\em lexicographic product}   $G \lexp G'$ is the graph with vertex set $V(G \lexp G')=V(G) \x V(G')$ and edge set 
 $E(G \lexp G')=\{ \{(x,x'),(y,y')\} \ | \ \{x, y\}  \in E(G), \mbox{ or } x=y \mbox{ and } \ \{x',y'\}  \in E(G') \}.$   
The next two results appeared in \cite{CDS} and \cite{I2009} respectively, where the authors used the notation $G[G']$ for  $G \lexp G'$.  %\iffalse In contrast to the restrictions for digraphs, it covers all graphs, and because the matrices are real symmetric, there is a basis of eigenvectors. \fi 
%It is worth noting that, for graphs $G$ and $G'$ with $G$ connected and order at least two, $\diam(G\lexp G')\le 2$.

\begin{theorem}
\label{thm:adjacency_lexprodgraphs}{\rm \cite[p. 72]{CDS}}
Let $G$ and $G'$ be graphs of orders $n$ and $n'$, respectively, such that $G'$ is $r'$-regular.
Let $\spec_{\mathcal{A}}(G)=(\alpha_1,\alpha_2,\dots,\alpha_n)$ and $\spec_{\mathcal{A}}(G')=(r',\alpha'_2,\dots,\alpha'_{n'})$.  
Then, \vspace{-5pt}
\[
\spec_{\mathcal{A}}(G{\lexp}G')=\left\{n'\alpha_i + r', \ i = 1, \dots, n \right\} \cup \left\{{\alpha'_j}^{(n)}, \ j = 2, \dots, n' \right\}.\vspace{-5pt}
\] 
\end{theorem}

\begin{theorem}\label{thm:lexprodgraphs}{\rm \cite%[Theorem 3.3]
{I2009}}  
Let $G$ and $G'$ be graphs of orders $n\geq 2$ and $n'$, respectively, such that $G$ is connected  and $G'$ is $r'$-regular. Let  $\dspec(G)=\{\dev_1, \dots, \dev_n\}$ and $\spec_\A(G')=\{r', \alpha'_2, \dots, \alpha'_{n'}\}$.  Then, \vspace{-5pt}
 \[\dspec(G \lexp G')=\{ n'\dev_i + 2n' -2-r', \ i = 1, \dots, n\}\cup\{ -(\alpha'_j +2)^{(n)}, \ j = 2, \dots, n'\}.\vspace{-5pt}\] 
 \end{theorem}

 To derive results on the  spectra of lexicographic products of digraphs, we  first investigate the spectra of the matrix product $M\lexp M'$ as defined in Section \ref{s:mtx-prod}.

\begin{theorem}
\label{thm:spectrum_lex_prod_general_matrices} Let $M\in\mathbb{R}^{n\times n}$ and  $M'\in\mathbb{R}^{n'\times n'}$ be irreducible nonnegative matrices such that $M'\bone_{n'}=\rho'\bone_{n'}$ for some $\rho'\in\mathbb{R}$.  Let $\spec(M)=\{\rho(M)=\lam_1,\lambda_2,\dots,\lambda_n\}$ and $\spec(M')=\{\rho',\lambda'_2,\dots,\lambda'_{n'}\}$. Then\vspace{-5pt}
\[
\spec(M\lexp M')= \left\{n'\lambda_i + \rho', \ i = 1, \dots, n \right\} \cup \left\{{\lambda'_j}^{(n)}, \ j = 2, \dots, n' \right\}.
\vspace{-5pt}\]
\end{theorem}
\bpf Choose $C$ such that $ C^{-1}MC=\mtx{\lam_1 & \boldzero^T \\ \boldzero & R}=\JCF_{M}$  where the diagonal elements of $R$ are $\lam_2,\dots,\lam_{n}$.  Use Proposition \ref{prop:MJ-JCF} to choose $C'$ such that $C'^{-1}\J_{n'} C'=\mtx{n' & \boldzero^T \\ \boldzero & O}=\diag(n',0,\dots,0)$ and $ C'^{-1}M'C'=\mtx{\rho' & \bx'^T \\ \boldzero & R'}$ where $\bx'\in\R^{n'-1}$ and  $R'$ is the part of $\JCF_{M'}$ associated with eigenvalues $\lam'_2,\dots,\lam'_{n'}$, all of which differ from $\rho'$. %; let $D=\diag(n',0,\dots,0)$.  
 Then
$(C^{-1}\otimes C'^{-1})(M\lexp M')(C\otimes C')=$
\renewcommand{\arraystretch}{1.3}
$\mtx{\lam_1 & \boldzero^T \\ \boldzero & R}\otimes \diag(n',0,\dots,0)+\I_n\otimes \mtx{\rho' & \bx'^T \\ \boldzero & R'}=$\\
\[{\scriptsize \lb \begin{array}{cc|cc|cc|c|cc} 
 \lam_1 n'& \boldzero^T & 0 & \boldzero^T& 0 & \boldzero^T & \cdots & 0 &\boldzero^T \\ 
\boldzero & O & \boldzero & O & \boldzero & O& \cdots & \boldzero & O    \\
\hline
 0 & \boldzero^T& \lam_2 n' & \boldzero^T & r_{12}n' & \boldzero^T &  \cdots & 0 &\boldzero^T \\ 
\boldzero & O & \boldzero & O & \boldzero & O & \cdots & \boldzero & O  \\ 
\hline
 0 & \boldzero^T  & 0 & \boldzero^T & \lam_3 n' & \boldzero^T&  \cdots & 0 &\boldzero^T \\ 
\boldzero & O & \boldzero & O & \boldzero & O & \cdots & \boldzero & O  \\ 
\hline
\vdots & \vdots &\vdots &\vdots &\vdots &\vdots & \ddots &\vdots &\vdots\\
\hline
%
% 0 & \boldzero^T  & 0 & \boldzero^T & 0 & \boldzero^T&  \cdots & r_{n-1,n}n' &\boldzero^T \\ 
%\boldzero & O & \boldzero & O & \boldzero & O & \cdots & \boldzero & O  \\ 
%\hline
%
 0 & \boldzero^T  & 0 & \boldzero^T & 0 & \boldzero^T&  \cdots & \lam_n n' &\boldzero^T \\ 
\boldzero & O & \boldzero & O & \boldzero & O & \cdots & \boldzero & O  \\ 
\end{array}\rb +  
%% ++
\lb \begin{array}{cc|cc|cc|c|cc} 
\rho' & \bx'^T & 0 & \boldzero^T& 0 & \boldzero^T & \cdots & 0 &\boldzero^T \\ 
\boldzero & R' & \boldzero & O & \boldzero & O& \cdots & \boldzero & O    \\
\hline
 0 & \boldzero^T& \rho' & \bx'^T& 0 & \boldzero^T &  \cdots & 0 &\boldzero^T \\ 
\boldzero & O & \boldzero & R' & \boldzero & O & \cdots & \boldzero & O  \\ 
\hline
 0 & \boldzero^T  & 0 & \boldzero^T & \rho' & \bx'^T &  \cdots & 0 &\boldzero^T \\ 
\boldzero & O & \boldzero & O & \boldzero & R' & \cdots & \boldzero & O  \\ 
\hline
\vdots & \vdots &\vdots &\vdots &\vdots &\vdots & \ddots &\vdots &\vdots\\
\hline
%
% 0 & \boldzero^T  & 0 & \boldzero^T & 0 & \boldzero^T&  \cdots & 0 &\boldzero^T \\ 
%\boldzero & O & \boldzero & O & \boldzero & O & \cdots & \boldzero & O  \\ 
%\hline
%
 0 & \boldzero^T  & 0 & \boldzero^T & 0 & \boldzero^T&  \cdots & \rho' & \bx'^T \\ 
\boldzero & O & \boldzero & O & \boldzero & O & \cdots & \boldzero & R'  
\end{array}\rb=
}\]
%%==
\[{\scriptsize \lb \begin{array}{cc|cc|cc|c|cc} 
\lam_1 n' +\rho'& \bx'^T & 0 & \boldzero^T& 0 & \boldzero^T & \cdots & 0 &\boldzero^T \\ 
\boldzero & R' & \boldzero & O & \boldzero & O& \cdots & \boldzero & O    \\
\hline
 0 & \boldzero^T& \lam_2n'+\rho' & \bx'^T & r_{12}n' & \boldzero^T &  \cdots & 0 &\boldzero^T \\ 
\boldzero & O & \boldzero & R' & \boldzero & O & \cdots & \boldzero & O  \\ 
\hline
 0 & \boldzero^T  & 0 & \boldzero^T & \lam_3n'+\rho' & \bx'^T&  \cdots & 0 &\boldzero^T \\ 
\boldzero & O & \boldzero & O & \boldzero & R' & \cdots & \boldzero & O  \\ 
\hline
\vdots & \vdots &\vdots &\vdots &\vdots &\vdots & \ddots &\vdots &\vdots\\
\hline
%
% 0 & \boldzero^T  & 0 & \boldzero^T & 0 & \boldzero^T&  \cdots & r_{n-1,n}n' &\boldzero^T \\ 
% \boldzero & O & \boldzero & O & \boldzero & O & \cdots & \boldzero & O  \\ 
%\hline
%
 0 & \boldzero^T  & 0 & \boldzero^T & 0 & \boldzero^T&  \cdots & \lam_nn'+\rho' &\bx'^T \\ 
\boldzero & O & \boldzero & O & \boldzero & O & \cdots & \boldzero & R'  \\ 
\end{array}\rb}\!.
\]
\renewcommand{\arraystretch}{1.3}
Since $(C^{-1}\otimes C'^{-1})(M\lexp M')(C\otimes C')$ is an upper triangular matrix,  the multiset of its diagonal elements is $\spec(M\lexp M')$.  %Since  $\diag(R')=(\lam'_2,\dots,\lam'_{n'})$, t
The multiset of diagonal elements is  $\left\{n'\lambda_i + \rho', \ i = 1, \dots, n \right\} \cup \left\{{\lambda'_j}^{(n)}, \ j = 2, \dots, n' \right\}$.
\epf

Even if $M$ and $M'$ are diagonalizable, it need not be the case that $M\lexp M'$ is diagonalizable, as the next example shows. 

\begin{example} \label{example_lexp_non_diag}
% does not work in sage.
% mathematics unable to find jcf but we can by doing ranks.
% see file not-enough-evecs
{\rm Consider the matrices
\[
M=
\begin{bmatrix}
0 & \frac{1}{3}(28-\sqrt{7})\\
\frac{1}{3}(28-\sqrt{7}) & 0
\end{bmatrix},\quad
M'=
\begin{bmatrix}
12 & 6 & 12\\
7 & 13 & 10\\
6 & 15 & 9
\end{bmatrix}
\]
and observe that they are both irreducible nonnegative matrices, and $M'\bone_3=30\,\bone_3$. Since $\spec(M)=\left\{\frac{1}{3}(28-\sqrt{7}),-\frac{1}{3}(28-\sqrt{7})\right\}$ and $\spec(M')=\left\{30,2+\sqrt{7},2-\sqrt{7}\right\}$, we see that both $M$ and $M'$ are diagonalizable. However, one finds that
\[
\JCF_{M\lexp M'}=
\begin{bmatrix}
 58-\sqrt{7} & 0 & 0 & 0 & 0 & 0 \\
 0 & 2+\sqrt{7} & 1 & 0 & 0 & 0 \\
 0 & 0 & 2+\sqrt{7} & 0 & 0 & 0 \\
 0 & 0 & 0 & 2+\sqrt{7} & 0 & 0 \\
 0 & 0 & 0 & 0 & 2-\sqrt{7} & 0 \\
 0 & 0 & 0 & 0 & 0 & 2-\sqrt{7}
\end{bmatrix}\!\!,
\]
which means that $M\lexp M'$ is not diagonalizable.
}
\end{example}

Based on the apparent anomaly of Example \ref{example_lexp_non_diag}, we now investigate the geometric multiplicities of the eigenvalues of $M\lexp M'$.
\begin{theorem}
\label{thm_geometric_mult_lex_product_matrices}
Let $M\in\mathbb{R}^{n\times n}$, $M'\in\mathbb{R}^{n'\times n'}$ be irreducible nonnegative matrices such that $M'\bone_{n'}=\rho'\bone_{n'}$ for some $\rho'\in\mathbb{R}$. Given $z\in \mathbb{C}$, define $\tilde z=\frac{z-\rho'}{n'}$,  %$a=\mult_M(\tilde z),\; 
$g=\gmult_{M}(\tilde z)$ %,\; a'=\mult_{M'}(z),\; 
and $g'=\gmult_{M'}(z)$. Then

\begin{equation}
\label{eqn_geom_mult_description_1551_10_nov}
\gmult_{M\lexp M'}(z)=
\begin{array}{l}
\left\{\begin{aligned}
&g &\mbox{if }\;z\not\in\spec(M')\setminus\{\rho'\},\;\tilde z\in\spec(M);\\
&ng' &\mbox{if }\; z\in\spec(M')\setminus\{\rho'\},\;\tilde z\not\in\spec(M);\\
&ng'+g &\mbox{if }\;z\in\spec(M')\setminus\{\rho'\},\;\tilde z\in\spec(M),\;ES_{M'}(z)\perp \bone_{n'} ;\\
&{ng'} &\mbox{if }\;z\in\spec(M')\setminus\{\rho'\},\;\tilde z\in\spec(M),\;ES_{M'}(z)\not\perp \bone_{n'} ;\\
&0 &\mbox{otherwise}.
\end{aligned}\right.
\end{array}
\end{equation}
\end{theorem}
\begin{proof} The eigenvalues of $M\lexp M'$ take two forms: 
$n'\lambda + \rho'$ for $\lam\in\spec(M)$ and $n$ copies of $\lambda'$ for $ \lam'\in\spec(M')$ and $\lam'\ne\rho'$.
Observe that $z=n'\tilde z +\rho'$, so $z$ takes the first form if and only if $\tilde z\in\spec(M)$. %The parameter $\tilde z$ serves to connect the two: $z=n'\tilde z +\rho'$, so $\tilde z\in\spec(M)$ if and only if $z=n'\lambda + \rho'$ for $\lam\in\spec(M)$.  
 The last case in \eqref{eqn_geom_mult_description_1551_10_nov} is thus immediate.   The first two cases in \eqref{eqn_geom_mult_description_1551_10_nov} concern the situation in which   there is no overlap between the values of the two forms.  Consider the structure of the matrix $(C^{-1}\otimes C'^{-1})(M\lexp M')(C\otimes C')$ as given in the proof of Theorem \ref{thm:spectrum_lex_prod_general_matrices}. Then these two cases are  a consequence of Lemma \ref{lem_sylvester_1645_10nov} after a suitable permutation of the rows and columns of $(C^{-1}\otimes C'^{-1})(M\lexp M')(C\otimes C')$.  
 
 The remaining two cases happen when  $z\in\spec(M')\setminus\{\rho'\}$ and $\tilde z\in\spec(M)$, so $z =\lam'=n'\lam+\rho'$ for $\lam\in\spec(M)$, $ \lam'\in\spec(M')$, and $\rho'\ne \lam'$. 
Let $V$ be a matrix of generalized eigenvectors for $M'$ corresponding to $\JCF_{M'}$, and define the vector $\ba=[a_i]\in\mathbb{R}^{n'}$ by \vspace{-5pt}
\[
a_i=
\begin{array}{l}
\left\{\begin{aligned}
&0 &\mbox{if }\; V\textbf{e}_i\perp\bone_{n'};\\
&1 &\mbox{if }\; V\textbf{e}_i\not\perp\bone_{n'}.
\end{aligned}
\right.
\end{array}
\vspace{-5pt}\]
We  rescale the columns of $V$ in such a way that $\bone_{n'}^TV=n'\ba^T$. Notice that  this implies  $V\textbf{e}_1=\bone_{n'}$.  Let $\J'=\J_{n'}$.  We claim that  $C'=V-\frac{1}{n'}\J' V+\bone_{n'}\textbf{e}_1^T$  satisfies the requirements for $C'$ in the proof of Theorem \ref{thm:spectrum_lex_prod_general_matrices}.  Furthermore, we claim  the first row of $C'^{-1}M' C'$  is  $\rho'\textbf{e}_1^T+\ba^T\JCF_{M'} -\rho'\ba^T$. For convenience, we define $\hat\bx=\JCF_{M'}^T\ba-\rho'\ba$. Observe that the first entry of $\hat \bx$ is zero, since $\hat\bx^T\textbf{e}_1=\ba^T\JCF_{M'}\textbf{e}_1
-\rho'\ba^T\textbf{e}_1=\rho'\ba^T\textbf{e}_1-\rho'\ba^T\textbf{e}_1=0$. Therefore, $\hat\bx^T=[0\ \bx'^T]$ in the notation of  the proof of Theorem \ref{thm:spectrum_lex_prod_general_matrices}. 
%{\red fix $\hat \bx$}  
First, we show that $C'$ is invertible: \vspace{-5pt}% We have that
\[
C'\textbf{e}_1=V\textbf{e}_1-\frac{1}{n'}\J' V\textbf{e}_1+\bone_{n'}\textbf{e}_1^T\textbf{e}_1=\bone_{n'}
-\bone_{n'}+\bone_{n'}=\bone_{n'}=V\textbf{e}_1
\vspace{-5pt}
\] 
and, for $i=2,\dots,n'$, \vspace{-5pt}
\[
C'\textbf{e}_i=V\textbf{e}_i-\frac{1}{n'}\J' V\textbf{e}_i+\bone_{n'}\textbf{e}_1^T\textbf{e}_i=
V\textbf{e}_i-\frac{1}{n'}\bone_{n'}n'\ba^T\textbf{e}_i=
V\textbf{e}_i-a_iV\textbf{e}_1,
\vspace{-5pt}
\]
so that $C'\textbf{e}_i$ is obtained from $V\textbf{e}_i$ by adding a scalar multiple of $V\textbf{e}_1$. Hence, $\det(C')=\det(V)\neq 0$. 

Moreover,
\[
\begin{aligned}
C'(n'\textbf{e}_1\textbf{e}_1^T)&=n'V\textbf{e}_1\textbf{e}_1^T-\J' V\textbf{e}_1\textbf{e}_1^T+n'\bone_{n'}\textbf{e}_1^T\textbf{e}_1\textbf{e}_1^T=
n'\bone_{n'}\textbf{e}_1^T-n'\bone_{n'}\textbf{e}_1^T+n'\bone_{n'}\textbf{e}_1^T\\
&=n'\bone_{n'}\textbf{e}_1^T=\J' V-\J' V+n'\bone_{n'}\textbf{e}_{1}^T=\J' V-\frac{1}{n'}\J'^2V+\J'\bone_{n'}\textbf{e}_1^T=\J' C'
\end{aligned}
\]
so $C'^{-1}\J_{n'} C'=\mtx{n' & \boldzero^T \\ \boldzero & O}$. Finally,
\[
\begin{aligned}
M'C'&=M'V-\frac{1}{n'}M'\J'V+M'\bone_{n'}\textbf{e}_1^T=V\JCF_{M'}-\frac{\rho'}{n'}\bone_{n'}\bone_{n'}^TV+\rho'\bone_{n'}\textbf{e}_1^T\\
&=V\JCF_{M'}-\rho'\bone_{n'}\ba^T+\rho'\bone_{n'}\textbf{e}_1^T.\\
%\end{aligned}\]
%\[\begin{aligned}
C'(\JCF_{M'}+\textbf{e}_1\hat\bx^T)&=\left(V-\frac{1}{n'}\J'V+\bone_{n'}\textbf{e}_1^T\right)(\JCF_{M'}+\textbf{e}_1\hat\bx^T) \\
&=V\JCF_{M'}-\frac{1}{n'}\J'V\JCF_{M'}+\bone_{n'}\textbf{e}_1^T\JCF_{M'}+V\textbf{e}_1\hat\bx^T-\frac{1}{n'}\J'V\textbf{e}_1\hat\bx^T+\bone_{n'}\textbf{e}_1^T\textbf{e}_1\hat\bx^T\\
%&=V\JCF_{M'}- \bone_{n'}\ba^T\JCF_{M'}+\rho'\bone_{n'}\textbf{e}_1^T +\bone_{n'}\hat\bx^T-\bone_{n'}\hat\bx^{T}+\bone_{n'}\hat\bx^T\\
&=V\JCF_{M'}- \bone_{n'}\ba^T\JCF_{M'}+\rho'\bone_{n'}\textbf{e}_1^T
+\bone_{n'}\hat\bx^T\\
&=M'C'+\rho'\bone_{n'}\ba^T- \bone_{n'}\ba^T\JCF_{M'}+\bone_{n'}\hat\bx^T\\
&=M'C'+\bone_{n'}\left(\rho'\ba^T-\ba^T\JCF_{M'}+\hat\bx^T\right)\\
&=M'C'+\bone_{n'}(-\hat\bx^T+\hat\bx^T)\\
&=M'C'.
\end{aligned}
\]
Therefore, $C'^{-1}M'C'=\JCF_{M'}+\textbf{e}_1\hat\bx^T$, and the claim is true. 

Let us now focus on the entries of $\hat\bx = [\hat x_i]$. We have already noticed that $\hat{x}_1=0$. Furthermore, for $i=2,\dots,n'$, we have that $\hat{x}_{i}=\hat\bx^{T}\textbf{e}_i=\ba^T\JCF_{M'}\textbf{e}_i-
\rho'\ba^T\textbf{e}_i=\lambda' a_i+\delta_i a_{i-1}-\rho'a_i$, where 
$\lambda'=(\JCF_{M'})_{ii}$, $\delta_i=0$ if $V\textbf{e}_i$ is an eigenvector of $M'$, and $\delta_i=1$ otherwise. Suppose now that $ES_{M'}(\lambda')\perp \bone_{n'}$. Then, whenever $\delta_i=0$ with $\lambda'=(\JCF_{M'})_{ii}$, $V\textbf{e}_i\perp \bone_{n'}$, so  $a_i=0$ and $\hat{x}_i=0$. On the other hand, if $ES_{M'}(\lambda')\not\perp \bone_{n'}$, we can find some $i$ such that $\lambda'=(\JCF_{M'})_{ii}$, $\delta_i=0$, and $a_i=1$, which means that $\hat{x}_i=\lambda'-\rho'\neq 0$.
%{\red [what if there are two eigenvectors for $\lam'$, with one $\perp\bone_{n'}$ and the other not?]}

 Take $z\in\mathbb{C}$ and suppose that $z\in\spec(M')\setminus\{\rho'\}$ and $\tilde z\in\spec(M)$. Define $u=\mult_M(\tilde z)$ and $ u'=\mult_{M'}(z)$. We can permute the rows and columns of $(C^{-1}\otimes C'^{-1})(M\lexp M')(C\otimes C')$ in such a way that all the appearances of $z$ on the diagonal are grouped together in a square block $B$. By virtue of Lemma \ref{lem_sylvester_1645_10nov}, the Jordan blocks relative to the eigenvalue $z$ only depend on $B$. We observe that $B$ has order $t=nu'+u$. Hence, $\gmult_{M\lexp M'}(z)=t-\rank(B-z\I_t)$. If $ES_{M'}(z)\perp \bone_{n'}$, from the discussion above we see that the entries in $\hat\bx$ do not influence the rank of $B-z\I_t$, since they can be reduced to zero by subtracting suitable rows of $B-z\I_t$. As a consequence, $\rank(B-z\I_t)=n(u'-g')+u-g$ and hence,
\[
\gmult_{M\lexp M'}(z)=nu'+u-nu'+ng'-u+g=ng'+g.
\]
If $ES_{M'}(z)\not\perp \bone_{n'}$, on the other hand, again using the discussion above we see that $\rank(B-z\I_t)=n(u'-g')+u$. Indeed, in this case, there exists $i\in\{2,\dots,n'\}$ such that $z=(\JCF_{M'})_{ii}$, $\delta_i=0$, and $\hat{x}_i\neq 0$. Therefore, every row of $B-z\I_t$ containing $\hat{\textbf{x}}^T$ is linearly independent from the remaining rows of $B-z\I_t$, and, thus, it increases the rank by $1$. This yields  
\[
\gmult_{M\lexp M'}(z)=nu'+u-nu'+ng'-u=ng'.
\]
\end{proof}

\begin{example}{\rm 
We now test Theorem \ref{thm_geometric_mult_lex_product_matrices} on the matrices $M$ and $M'$ defined in Example \ref{example_lexp_non_diag}. As predicted by Theorem \ref{thm:spectrum_lex_prod_general_matrices} we have that 
\[
\spec(M\lexp M')=\{58-\sqrt{7},(2+\sqrt{7})^{(3)},(2-\sqrt{7})^{(2)}\}.
\]
\begin{itemize}
\item
If $z=58-\sqrt{7}$ then $\tilde z=\frac{1}{3}(28-\sqrt{7})$. This corresponds to the first case of \eqref{eqn_geom_mult_description_1551_10_nov}, and hence we obtain $\gmult_{M\lexp M'}(58-\sqrt{7})=\gmult_M(\frac{1}{3}(28-\sqrt{7}))=1$.
\item
If $z=2-\sqrt{7}$ then $\tilde z=\frac{1}{3}(-28-\sqrt{7})$. This corresponds to the second case of \eqref{eqn_geom_mult_description_1551_10_nov}, and hence we obtain $\gmult_{M\lexp M'}(2-\sqrt{7})=n\,\gmult_{M'}(2-\sqrt{7})=2$.
\item
If $z=2+\sqrt{7}$ then $\tilde z=-\frac{1}{3}(28-\sqrt{7})$. Moreover, we find that $ES_{M'}(2+\sqrt{7})=\operatorname{span}(\textbf{v})$ with $\textbf{v}^T=\begin{bmatrix}
\frac{24-4\sqrt{7}}{-25+7\sqrt{7}}
&
\frac{30-28\sqrt{7}}{-201+45\sqrt{7}}
&
1
\end{bmatrix}
$.
Since $\textbf{v}^T\bone_3=\frac{13-\sqrt{7}}{-75+21\sqrt{7}}\neq 0$, we see that $ES_{M'}(2+\sqrt{7})\not\perp\bone_3$, so that this  corresponds to the fourth case of \eqref{eqn_geom_mult_description_1551_10_nov}, and hence we obtain $\gmult_{M\lexp M'}(2+\sqrt{7})=n\,\gmult_{M'}(2+\sqrt{7})=2$.
\end{itemize}
Notice that $\gmult_{M\lexp M'}(2+\sqrt{7})<\mult_{M\lexp M'}(2+\sqrt{7})$, which implies that $M\lexp M'$ is not diagonalizable (as computed in Example \ref{example_lexp_non_diag}).
}\end{example}

We can apply Theorem \ref{thm:spectrum_lex_prod_general_matrices} and Theorem \ref{thm_geometric_mult_lex_product_matrices} to derive results on the adjacency spectra of lexicographic products of digraphs, and for the Laplacian and signless Laplacian spectra of lexicographic products of digraphs with additional conditions. The first part of the following result was proved in \cite{EH80} for the case where $\dig'$ is a regular digraph (all row and column sums of its adjacency matrix are equal), and is an extension of results known to hold for graphs (see. e.g., \cite{BKPS}).

\begin{corollary}\label{thm:TRlexprod-dig_new_new} Let $\dig$ and $\dig'$ be strongly connected digraphs of orders $n$ and $n'$, respectively,  such that 
 $\dig'$ is $r'$-out-regular.
Let $\spec_{\A}(\dig)=(\alpha_1,\alpha_2,\dots,\alpha_n)$ and $\spec_{\A}(\dig')=(r',\alpha'_2,\dots,\alpha'_{n'})$.  
Then,
\[
\spec_{\A}(\dig{\lexp}\dig')=\left\{n'\alpha_i + r', \ i = 1, \dots, n \right\} \cup \left\{{\alpha'_j}^{(n)}, \ j = 2, \dots, n' \right\}.
\]
Given $z\in \mathbb{C}$, define $\tilde z=\frac{z-r'}{n'}$,  %$a=\mult_{\A(\dig)}(\tilde z),$ 
$g=\gmult_{\A(\dig)}(\tilde z)$, and %,\; a'=\mult_{\A(\dig')}(z),\; 
$g'=\gmult_{\A(\dig')}(z)$. Then

\[
\gmult_{\A(\dig\lexp\dig')}(z)=
\begin{array}{l}
\left\{\begin{aligned}
&g &\mbox{if }\;z\not\in\spec_{\A}(\dig')\setminus\{r'\},\;\tilde z\in\spec_{\A}(\dig);\\
&ng' &\mbox{if }\; z\in\spec_{\A}(\dig')\setminus\{r'\},\;\tilde z\not\in\spec_{\A}(\dig);\\
&ng'+g &\mbox{if }\;z\in\spec_{\A}(\dig')\setminus\{r'\},\;\tilde z\in\spec_{\A}(\dig),\;ES_{\A(\dig')}(z)\perp \bone_{n'} ;\\
&ng' &\mbox{if }\;z\in\spec_{\A}(\dig')\setminus\{r'\},\;\tilde z\in\spec_{\A}(\dig),\;ES_{\A(\dig')}(z)\not\perp \bone_{n'} ;\\
&0 &\mbox{otherwise}.
\end{aligned}\right.
\end{array}
\]
\end{corollary}

If both $\dig$ and $\dig'$ are out-regular, then $\dig\lexp\dig'$ is out-regular, too. As a consequence, both the spectrum of the Laplacian matrix and of the signless Laplacian matrix of $\dig\lexp\dig'$ are obtained via shifting the spectrum of its adjacency matrix. Corollary \ref{cor_spectrum_laplacian_lexp} and Corollary \ref{cor_spectrum_SIGNLESS_laplacian_lexp} are then derived from Corollary \ref{thm:TRlexprod-dig_new_new} using basic algebraic manipulations.

\begin{corollary}\label{cor_spectrum_laplacian_lexp} Let $\dig$ and $\dig'$ be strongly connected digraphs of orders $n$ and $n'$, respectively, such that $\dig$ is $r$-out-regular and $\dig'$ is $r'$-out-regular. Let $\spec_L(\dig)=(0,\alpha^L_2,\dots,\alpha^L_n)$ and $\spec_L(\dig')=(0,\alpha^{L'}_2,\dots,\alpha^{L'}_{n'})$.  
Then $\dig\lexp\dig'$ is $(rn'+r')$-out-regular and 
\[
\spec_L(\dig{\lexp}\dig')=\{0\}\cup\{n'\alpha^L_i, \ i=2,\dots,n\}\cup \{(\alpha^{L'}_{j}+rn')^{(n)}, \ j=2,\dots,n'\}.           
\]
Given $z\in \mathbb{C}$, define $\tilde z=\frac{z}{n'}$, $\hat z=z-rn'$,  %$a=\mult_{L(\dig)}(\tilde z),\; 
$g=\gmult_{L(\dig)}(\tilde z)$, %,\; a'=\mult_{L(\dig')}(\hat z),\; 
and $g'=\gmult_{L(\dig')}(\hat z)$. Then
\[
\gmult_{L(\dig\lexp\dig')}(z)=
\begin{array}{l}
\left\{\begin{aligned}
&g &\mbox{if }\;\hat z\not\in\spec_{L}(\dig')\setminus\{0\},\;\tilde z\in\spec_{L}(\dig);\\
&ng' &\mbox{if }\; \hat z\in\spec_{L}(\dig')\setminus\{0\},\;\tilde z\not\in\spec_{L}(\dig);\\
&ng'+g &\mbox{if }\;\hat z\in\spec_{L}(\dig')\setminus\{0\},\;\tilde z\in\spec_{L}(\dig),\;ES_{L(\dig')}(\hat z)\perp \bone_{n'} ;\\
&ng' &\mbox{if }\;\hat z\in\spec_{L}(\dig')\setminus\{0\},\;\tilde z\in\spec_{L}(\dig),\;ES_{L(\dig')}(\hat z)\not\perp \bone_{n'} ;\\
&0 &\mbox{otherwise}.
\end{aligned}\right.
\end{array}
\]
\end{corollary}
\begin{corollary}\label{cor_spectrum_SIGNLESS_laplacian_lexp} Let $\dig$ and $\dig'$ be strongly connected digraphs of orders $n$ and $n'$, respectively, such that $\dig$ is $r$-out-regular and $\dig'$ is $r'$-out-regular. Let $\spec_Q(\dig)=(\alpha^Q_1,\alpha^Q_2,\dots,\alpha^Q_n)$ and $\spec_Q(\dig')=(\alpha^{Q'}_1,\alpha^{Q'}_2,\dots,\alpha^{Q'}_{n'})$. Then
\[
\spec_Q(\dig{\lexp}\dig')=\{n'\alpha^Q_i+2r', \ i=1,\dots,n\}\cup\{(\alpha^{Q'}_j+rn')^{(n)}, \ j=2,\dots,n'\}.           
\]
Given $z\in \mathbb{C}$, define $\tilde z=\frac{z-2r'}{n'}$, $\hat z=z-rn'$,  %$a=\mult_{Q(\dig)}(\tilde z),\; 
$g=\gmult_{Q(\dig)}(\tilde z)$, %$,\; a'=\mult_{Q(\dig')}(\hat z),\; 
and $g'=\gmult_{Q(\dig')}(\hat z)$. Then
\[
\gmult_{Q(\dig\lexp\dig')}(z)=
\begin{array}{l}
\left\{\begin{aligned}
&g &\mbox{if }\;\hat z\not\in\spec_{Q}(\dig')\setminus\{2r'\},\;\tilde z\in\spec_{Q}(\dig);\\
&ng' &\mbox{if }\; \hat z\in\spec_{Q}(\dig')\setminus\{2r'\},\;\tilde z\not\in\spec_{Q}(\dig);\\
&ng'+g &\mbox{if }\;\hat z\in\spec_{Q}(\dig')\setminus\{2r'\},\;\tilde z\in\spec_{Q}(\dig),\;ES_{Q(\dig')}(\hat z)\perp \bone_{n'} ;\\
&ng' &\mbox{if }\;\hat z\in\spec_{Q}(\dig')\setminus\{2r'\},\;\tilde z\in\spec_{Q}(\dig),\;ES_{Q(\dig')}(\hat z)\not\perp \bone_{n'} ;\\
&0 &\mbox{otherwise}.
\end{aligned}\right.
\end{array}
\]
\end{corollary}
The lexicographic product $\dig\lexp \dig'$ is strongly connected if and only if $\dig$ is strongly connected \cite{H18}, but $\dig'$ need  not be. If $\dig'$ is not strongly connected, then $d_{\dig'}(x',y')=\infty$ when there is no dipath from $x'$ to $y'$. Due to this subtlety, in this section only we list any requirements for strong connectivity explicitly.  %Thus when discussing distance matrices of $\dig\lexp \dig'$, we assume $\dig$ is strongly connected.  %The next formula for the distance between two vertices in $\dig \lexp \dig'$ can be found in \cite{H18}. 
For a vertex $x$ of a strongly connected digraph $\dig$,  $\xi_\dig(x)$ is the length of a shortest (nontrivial)  dicycle  containing $x$.  If  $\dig$ has at least one dicycle, the minimum length of a dicycle in $\dig$ is called the {\em girth} of  $\dig$, denoted $g(\dig)$.

\begin{proposition}\label{distlex} {\rm \cite{H18}} %[Proposition 10.2.4]
 If $\dig, \dig'$ are digraphs such that $\dig$ is strongly connected, the distance formula for the lexicographic product $\dig \lexp \dig'$ is
\[
d_{\dig \lexp \dig'}((x,x'),(y,y')) = \left\{ \begin{array}{lr} d_{\dig}(x,y) &  if \ x \ne y\\
\min \{\xi_\dig(x), \ d_{\dig'}(x',y')\} & if \ x = y.
\end{array} \right.
\]
\end{proposition}

\begin{obs}\label{obs:lex-longcycle}
 If $\dig$ and $ \dig'$ are strongly connected digraphs such that   $\diam \dig' \leq \ds g(\dig)$,  
% $d_{\dig'}(x',y') \leq \ds \min_{x \in V(\dig)} \xi_\dig(x)$  for all $x', y' \in V(\dig')$, 
 then 
the distance formula in Proposition \ref{distlex} becomes
\[
d_{\dig \lexp \dig'}((x,x'),(y,y')) = \left\{ \begin{array}{lr} d_{\dig}(x,y) &  if \ x \ne y\\
 d_{\dig'}(x',y') & if \ x = y.
\end{array} \right.
\]
In this case, by a suitable ordering of vertices, the distance matrix $\D(\dig \lexp \dig')$ can be written in the form
$\D(\dig \lexp \dig') = \D(\dig) \otimes \J_{n'} + \I_n \otimes \D(\dig')=\D(\dig)\lexp\D(\dig')$.
\end{obs}

The {\em complement} of a digraph $\dig=(V,E)$ is the digraph $\OL\dig=(V,\OL E)$ where $\OL E$ consists of all  arcs not in $\dig$. \vspace{-5pt}
\begin{obs}\label{obs:lex-doubly-directed}
 If $\dig$ and $ \dig'$ are digraphs such that $\dig$ is strongly connected and every vertex is incident with a doubly directed arc, then $\xi_\dig(x) = 2$ for any vertex $x$ of $\dig$. In this case, by a suitable ordering of vertices, the distance matrix $\D(\dig \lexp \dig')$ can be written in the form
$\D(\dig \lexp \dig') = \D(\dig) \otimes \J_{n'} + \I_n \otimes (\A(\dig') + 2\A(\overline{\dig'}))=\D(\dig)\lexp (\A(\dig') + 2\A(\overline{\dig'}))$
as derived in \cite{I2009} for graphs.

\end{obs}

We can apply Theorem \ref{thm:spectrum_lex_prod_general_matrices} and Theorem \ref{thm_geometric_mult_lex_product_matrices} to provide results on the distance spectra of lexicographic products of digraphs which satisfy certain hypotheses. 
The next result is an  immediate consequence of Observation \ref{obs:lex-longcycle}, Theorem \ref{thm:spectrum_lex_prod_general_matrices}, and Theorem \ref{thm_geometric_mult_lex_product_matrices}.

\begin{corollary}\label{cor_spectrum_distance_lexp_long_cycle} Let $\dig$ and $\dig'$ be strongly connected digraphs of orders $n$ and $n'$, respectively,  such that  $\dig'$ is $t'$-transmission regular, and $\diam \dig' \leq \ds  g(\dig)$.
Let $\spec_{\D}(\dig)=(\dev_1,\dev_2,\dots,\dev_n)$ and $\spec_{\D}(\dig')=(t',\dev'_2,\dots,\dev'_{n'})$.  
Then
\[
\spec_{\D}(\dig{\lexp}\dig')=\left\{n'\dev_i + t', \ i = 1, \dots, n \right\} \cup \left\{{\dev'_j}^{(n)}, \ j = 2, \dots, n' \right\}.
\]
Given $z\in \mathbb{C}$, define $\tilde z=\frac{z-t'}{n'}$,  %$a=\mult_{\D(\dig)}(\tilde z),\;
$ g=\gmult_{\D(\dig)}(\tilde z)$, %,\; a'=\mult_{\D(\dig')}(z),\; 
and $g'=\gmult_{\D(\dig')}(z)$. Then

\[
\gmult_{\D(\dig\lexp\dig')}(z)=
\begin{array}{l}
\left\{\begin{aligned}
&g &\mbox{if }\;z\not\in\spec_{\D}(\dig')\setminus\{t'\},\;\tilde z\in\spec_{\D}(\dig);\\
&ng' &\mbox{if }\; z\in\spec_{\D}(\dig')\setminus\{t'\},\;\tilde z\not\in\spec_{\D}(\dig);\\
&ng'+g &\mbox{if }\;z\in\spec_{\D}(\dig')\setminus\{t'\},\;\tilde z\in\spec_{\D}(\dig),\;ES_{\D(\dig')}(z)\perp \bone_{n'} ;\\
&ng' &\mbox{if }\;z\in\spec_{\D}(\dig')\setminus\{t'\},\;\tilde z\in\spec_{\D}(\dig),\;ES_{\D(\dig')}(z)\not\perp \bone_{n'} ;\\
&0 &\mbox{otherwise}.
\end{aligned}\right.
\end{array}
\]
\end{corollary}

To establish a result about the distance matrix of a lexicographic product when every vertex of the first factor is incident with a doubly directed arc, we make use of Observation \ref{obs:lex-doubly-directed}, Theorem \ref{thm:spectrum_lex_prod_general_matrices},  Theorem \ref{thm_geometric_mult_lex_product_matrices}, and the next proposition.

\begin{proposition}\label{o:adj-comp}
Let $\dig$ be an $r$-out-regular digraph with $\spec_\A(\dig)=\{r,\alpha_2,\dots,\alpha_{n}\}$ and let  $B= \A(\dig)+2\A(\overline{\dig})$.  Then $B$ is an irreducible nonnegative matrix, $\spec(B)=\{2n-2-r,-(\alpha_2+2),\dots,-(\alpha_{n}+2)\}$,  and $\rho(B)=2n-2-r$.  Furthermore, $\gmult_B(-\alpha_j-2)=\gmult_{\A(\dig)}(\alpha_j)$ for $\alpha_j\ne r$ and    $\gmult_B(-r-2)=\gmult_{\A(\dig)}(r)-1$. \vspace{-5pt}

Suppose $\bv_j$ is an eigenvector of $\A(\dig)$ for eigenvalue $\alpha_j$ for $j=2,\dots,k$, and define $\beta_j=\frac{2\bv_j^T\bone_{n}}{r-\alpha_j-2n}$.  Then $\bone_{n}$ is an eigenvector of $B$ for eigenvalue $2n-2-r$, and   $\bv_j+\beta_j\bone_{n}$ is an eigenvector of $B$ for eigenvalue $-\alpha_j-2$ for $j=2,\dots,k$.  
\end{proposition}
\bpf
Observe first that every off-diagonal entry of $B= \A(\dig)+2\A(\overline{\dig})$ is nonzero, so $B$ is an irreducible nonnegative matrix. Furthermore, $\A(\overline{\dig})=\J_{n}-\I_{n}-\A(\dig)$, so  $B= 2\J_{n}-2\I_{n}-\A(\dig)$. 
Hence,
\[
\begin{aligned}
B\bone_{n}=2\J_{n}\bone_{n}-2\I_{n}\bone_{n}-\A(\dig)\bone_{n}=(2n-2-r)\bone_{n}
\end{aligned}
\]
and $2n-2-r$ is the spectral radius of $B$.
Let $\JCF_{\A(\dig)}=\begin{bmatrix}
r & \by^T\\
\boldzero & R
\end{bmatrix}$. Apply  Proposition \ref{prop:MJ-JCF} to choose $C$ such that $C^{-1}\J_{n}C=\begin{bmatrix}
n & \boldzero^T\\
\boldzero & O
\end{bmatrix}$ and $C^{-1}\A(\dig)C=\begin{bmatrix}
r & \textbf{x}^T\\
\boldzero & R
\end{bmatrix}$
for some Jordan matrix $R$ and  $\textbf{x}\in \mathbb{R}^{n-1}$. Then
\begin{eqnarray}
C^{-1}BC&=&2C^{-1}\J_{n}C-2C^{-1}\I_{n}C-C^{-1}\A(\dig)C\nonumber\\
&=&\label{e:JCF-314}
\begin{bmatrix}
2n & \boldzero^T\\
\boldzero & O
\end{bmatrix}
-2\I_{n}
-
\begin{bmatrix}
r & \textbf{x}^T\\
\boldzero & R
\end{bmatrix}
=
\begin{bmatrix}
2n-2-r & -\textbf{x}^T \\
\boldzero & -2\I_{n-1}-R
\end{bmatrix}\!,
\end{eqnarray} %
which shows that $\spec(B)=\{2n-2-r,-(\alpha_2+2),\dots,-(\alpha_{n}+2)\}$.  
Since $B$ is irreducible, $2n-2-r$ is a simple eigenvalue of $B$. Applying Lemma \ref{lem_sylvester_1645_10nov} to \eqref{e:JCF-314}, we see that 
\[
\JCF_{B}=
\begin{bmatrix}
2n-2-r & \boldzero^T \\
\boldzero & -2\I_{n-1}-R
\end{bmatrix}
\]
so that $\gmult_{B}(-\alpha_j-2)=\gmult_{\A(\dig)}(\alpha_j)$, $j=2,\dots,n$ for $\alpha_j\ne r$ and $\gmult_B(-r-2)=\gmult_{\A(\dig)}(r)-1$.

Observe that $r-\alpha_j-2n\neq 0$ because  $|\alpha_j|\le r<n$, where the second inequality is due to the fact that $r$ is the out-degree of each vertex in $\dig$. Hence,
$
|r-\alpha_j-2n|\geq 2n-|r|-|\alpha_j|> 0
$ and $\beta_j$ is well defined.
It is immediate that $\bone_n$ is an eigenvector for $2n-2-r$, and
\[\begin{aligned}
B(\bv_j+\beta_j\bone_{n}) %&= &(2\J_{n}-2\I_{n}-\A(\dig))(\bv_j+\beta_j\bone_{n})\\
&=  2\bone_{n}^T\bv_j \bone_{n}+2n\beta_j\bone_{n} -2\bv_j-2\beta_j\bone_{n}-\alpha_j\bv_j-\beta_jr\bone_n\\
%&=&(-\alpha_j-2)\bv_j+(2\bone_{n}^T\bv_j+2n\beta_j-2\beta_j-r\beta_j)\bone_n\\
&= (-\alpha_j-2)\bv_j+\lp\frac{2\bone_{n}^T\bv_j}{\beta_j}+2n-2-r\rp\beta_j\bone_n\\
&= (-\alpha_j-2)(\bv_j+\beta_j\bone_{n}).
\end{aligned}\]
\epf

\begin{theorem}\label{thm:TRlexprod-dig_doubly_directed} 
Let $\dig$ and $\dig'$ be  strongly connected digraphs of orders $n$ and $n'$, respectively,   %{\red every vertex is incident with a doubly directed arc}  
such that   every vertex of $\dig$ is incident with a doubly directed arc, %{\red [do we need this to commute with $\J$?  Or is `every vertex is incident with a doubly directed arc' enough?]}
and all vertices in $\dig'$ have out-degree $r'$. Let $\spec_{\D}(\dig)=(\dev_1,\dev_2,\dots,\dev_n)$ and $\spec_{\A}(\dig')=(r',\alpha'_2,\dots,\alpha'_{n'})$.  
Then
\[
\spec_{\D}(\dig{\lexp}\dig')=\left\{n'\dev_i + 2n' - 2 - r', \ i = 1, \dots, n \right\} \cup \left\{{-(\alpha'_j+2)}^{(n)}, \ j = 2, \dots, n' \right\}\!.
\]
Given $z\in \mathbb{C}$, define $\tilde z=\frac{z-2n'+2+r'}{n'}$,  %$a=\mult_{\D(\dig)}(\tilde z),\; 
$g=\gmult_{\D(\dig)}(\tilde z)$, %,\; a'=\mult_{\A(\dig')}(-z-2),\; 
and $g'=\gmult_{\A(\dig')}(-z-2)$. Then
\[
\gmult_{\D(\dig \lexp \dig')}(z)=
\begin{array}{l}
\left\{\begin{aligned}
&g &\mbox{if }\;-z-2\not\in\spec_{\A}(\dig')\setminus\{r'\},\;\tilde z\in\spec_{\D}(\dig);\\
&ng' &\mbox{if }\; -z-2\in\spec_{\A}(\dig')\setminus\{r'\},\;\tilde z\not\in\spec_{\D}(\dig);\\
&ng'+g &\mbox{if }\;-z-2\in\spec_{\A}(\dig')\setminus\{r'\},\;\tilde z\in\spec_{\D}(\dig),\;ES_{\A(\dig')}(-z-2)\perp \bone_{n'} ;\\
&ng' &\mbox{if }\;-z-2\in\spec_{\A}(\dig')\setminus\{r'\},\;\tilde z\in\spec_{\D}(\dig),\;ES_{\A(\dig')}(-z-2)\not\perp \bone_{n'} ;\\
&0 &\mbox{otherwise}.
\end{aligned}\right.
\end{array}
\]
\end{theorem}

\bpf 
By Observation \ref{obs:lex-doubly-directed}, $\D(\dig \lexp \dig') = \D(\dig)\lexp (\A(\dig') + 2\A(\overline{\dig'}))$.
Let $M'= \A(\dig') + 2\A(\overline{\dig'})$, so $\spec(M')=\{2n'-2-r',-(\alpha_2'+2),\dots,-(\alpha_{n'}'+2)\}$ %\marginpar{\red Why is \\$2n'-2-r'$\\$=\rho(M')$?} 
by Proposition \ref{o:adj-comp}.
The first part of the theorem then follows from %Observation \ref{obs:lex-doubly-directed} and 
Theorem \ref{thm:spectrum_lex_prod_general_matrices}. 

Since $\dig'$ strongly connected, $r'$ is a simple eigenvalue and  $\gmult_{M'}(-\alpha'_j-2)=\gmult_{\A(\dig')}(\alpha'_j)$ for $j=2,\dots,n'$ by Proposition \ref{o:adj-comp}. We now claim that $ES_{M'}(-\alpha'_j-2)\perp \bone_{n'}$ exactly when $ES_{\A(\dig')}(\alpha'_j)\perp\bone_{n'}$. First, if $ES_{\A(\dig')}(\alpha'_j)\perp\bone_{n'}$, then given $\textbf{v}\in ES_{\A(\dig')}(\alpha'_j)$,
\[
M'\textbf{v}=2\J_{n'}\textbf{v}-2\I_{n'}\textbf{v}-\A(\dig')\textbf{v}=(-\alpha'_j-2)\textbf{v}
\]
so that $ES_{\A(\dig')}(\alpha'_j)\subseteq ES_{M'}(-\alpha'_j-2)$. Since $\gmult_{M'}(-\alpha'_j-2)=\gmult_{\A(\dig')}(\alpha'_j)$, we conclude that $ES_{\A(\dig')}(\alpha'_j)= ES_{M'}(-\alpha'_j-2)$ and the claim follows in this case. Suppose now that $ES_{\A(\dig')}(\alpha'_j)\not\perp\bone_{n'}$, and let $\textbf{w}\in ES_{\A(\dig')}(\alpha'_j)$, $\textbf{w}\not\perp\bone_{n'}$.
%Observe that $r'-\alpha'_j-2n'\neq 0$. %Indeed, we have that  $|\alpha'_j|\le r'<n'$, where the second inequality is due to the fact that $r'$ is the out-degree of each vertex in $\dig'$. Hence, $|r'-\alpha'_j-2n'|\geq 2n'-|r'|-|\alpha'_j|> 0.$
Define $\tilde{\textbf{w}}=\textbf{w}+\beta_j\bone_{n'}$ with $\beta_j=\frac{2\textbf{w}^T\bone_{n'}}{r'-\alpha'_j-2n'}$ as in Proposition \ref{o:adj-comp}, so that $\tilde{\textbf{w}}\in ES_{M'}(-\alpha'_j-2)$. The claim then follows since
\[
\begin{aligned}
\tilde{\textbf{w}}^T\bone_{n'}&=(\textbf{w}+\beta_j\bone_{n'})^T\bone_{n'}=\textbf{w}^T\bone_{n'}+n'\beta_j
=
\textbf{w}^T\bone_{n'}\left(1+\frac{2n'}{r'-\alpha'_j-2n'}\right)\\
&=
\textbf{w}^T\bone_{n'}
\left(
\frac{r'-\alpha'_j}{r'-\alpha'_j-2n'}
\right)\neq 0
\end{aligned}
\]
because $r'$ is a simple eigenvalue. The second part of the theorem is then a direct consequence of Theorem \ref{thm_geometric_mult_lex_product_matrices}. 
\epf

Using Corollary \ref{cor_spectrum_distance_lexp_long_cycle} and Theorem \ref{thm:TRlexprod-dig_doubly_directed} we obtain a description for the spectrum of the (signless) distance Laplacian matrix of $\dig\lexp\dig'$ under certain conditions. This is done in Corollary \ref{cor_spectrum_dist_lap_lex_long_cycle}, Corollary \ref{cor_spectrum_SIGNLESS_dist_lap_lex_long_cycle}, Corollary \ref{cor_spectrum_dist_lap_lex_doubly_dir} and Corollary \ref{cor_spectrum_dist_SIGNLESS_lap_lex_doubly_dir}.
\begin{corollary}
\label{cor_spectrum_dist_lap_lex_long_cycle}
Let $\dig$ and $\dig'$ be strongly connected digraphs of orders $n$ and $n'$, respectively,  such that $\dig$ is $t$-transmission regular, $\dig'$ is $t'$-transmission regular and $\diam \dig' \leq \ds  g(\dig)$.
Let $\spec_{\D^L}(\dig)=(0,\dlev_2,\dots,\dlev_n)$ and $\spec_{\D^L}(\dig')=(0,{\dlev}'_2,\dots,{\dlev}'_{n'})$.  
Then $\dig\lexp\dig'$ is $(tn'+t')$-transmission regular and
\[
spec_{\D^L}(\dig\lexp\dig')=\{0\}\cup\{n'\dlev_i, \ i=2,\dots,n\}\cup \{({{\dlev}'_j}+tn')^{(n)}, \ j=2,\dots,n'\}.
\]
Given $z\in \mathbb{C}$, define $\tilde z=\frac{z}{n'}$, $\hat z=z-tn'$,   $ g=\gmult_{\D^L(\dig)}(\tilde z),$ and $  g'=\gmult_{\D^L(\dig')}(\hat z)$. Then
\[
\gmult_{\D^L(\dig\lexp\dig')}(z)=
\begin{array}{l}
\left\{\begin{aligned}
&ng' &\mbox{if }\; \hat z\in\spec_{\D^L}(\dig')\setminus\{0\},\;\tilde z\not\in\spec_{\D^L}(\dig);\\
&g &\mbox{if }\;\hat z\not\in\spec_{\D^L}(\dig')\setminus\{0\},\;\tilde z\in\spec_{\D^L}(\dig);\\
&ng'+g &\mbox{if }\;\hat z\in\spec_{\D^L}(\dig')\setminus\{0\},\;\tilde z\in\spec_{\D^L}(\dig),\;ES_{\D^L(\dig')}(\hat z)\perp \bone_{n'} ;\\
&ng' &\mbox{if }\;\hat z\in\spec_{\D^L}(\dig')\setminus\{0\},\;\tilde z\in\spec_{\D^L}(\dig),\;ES_{\D^L(\dig')}(\hat z)\not\perp \bone_{n'} ;\\
&0 &\mbox{otherwise}.
\end{aligned}\right.
\end{array}
\]
\end{corollary}
\begin{corollary}
\label{cor_spectrum_SIGNLESS_dist_lap_lex_long_cycle}
Let $\dig$ and $\dig'$ be strongly connected digraphs of orders $n$ and $n'$, respectively,  such that $\dig$ is $t$-transmission regular, $\dig'$ is $t'$-transmission regular and $\diam \dig' \leq \ds  g(\dig)$.
Let $\spec_{\D^Q}(\dig)=(\dQev_1,\dQev_2,\dots,\dQev_n)$ and $\spec_{\D^Q}(\dig')=({\dQev}'_1,{\dQev}'_2,\dots,{\dQev}'_{n'})$.  
Then
\[
spec_{\D^Q}(\dig\lexp\dig')=\{n'\dQev_i+2t', \ i=1,\dots,n\}
\cup \{({{{\dQev}'_j}+tn'})^{(n)}, \ j=2,\dots,n'\}.
\]
Given $z\in \mathbb{C}$,  define $\tilde z=\frac{z-2t'}{n'}$, $\hat z=z-tn'$,   $ g=\gmult_{\D^Q(\dig)}(\tilde z)$, and $ g'=\gmult_{\D^Q(\dig')}(\hat z)$. Then

\[
\gmult_{\D^Q(\dig\lexp\dig')}(z)=
\begin{array}{l}
\left\{\begin{aligned}
&ng' &\mbox{if }\; \hat z\in\spec_{\D^Q}(\dig')\setminus\{2t'\},\;\tilde z\not\in\spec_{\D^Q}(\dig);\\
&g &\mbox{if }\;\hat z\not\in\spec_{\D^Q}(\dig')\setminus\{2t'\},\;\tilde z\in\spec_{\D^Q}(\dig);\\
&ng'+g &\mbox{if }\;\hat z\in\spec_{\D^Q}(\dig')\setminus\{2t'\},\;\tilde z\in\spec_{\D^Q}(\dig),\;ES_{\D^Q(\dig')}(\hat z)\perp \bone_{n'} ;\\
&ng' &\mbox{if }\;\hat z\in\spec_{\D^Q}(\dig')\setminus\{2t'\},\;\tilde z\in\spec_{\D^Q}(\dig),\;ES_{\D^Q(\dig')}(\hat z)\not\perp \bone_{n'} ;\\
&0 &\mbox{otherwise}.
\end{aligned}\right.
\end{array}
\]\end{corollary}
\begin{corollary}
\label{cor_spectrum_dist_lap_lex_doubly_dir}
Let $\dig$ and $\dig'$ be strongly connected digraphs of orders $n$ and $n'$, respectively, such that $\dig$ is $t$-transmission regular, every vertex of $\dig$ is incident with a doubly directed arc, and all vertices in $\dig'$ have out-degree $r'$. Let $\spec_{\D^L}(\dig)=(0,\dlev_2,\dots,\dlev_n)$ and $\spec_{\A}(\dig')=(r',\alpha'_2,\dots,\alpha'_{n'})$.  
Then $\dig\lexp\dig'$ is $(tn'+2n'-2-r')$-transmission regular and
\[
\spec_{\D^L}(\dig{\lexp}\dig')=\{0\}\cup\{n'\dlev_i, \ i=2,\dots,n\}\cup\{(tn'+2n'+\alpha'_j-r')^{(n)}, \ j=2,\dots,n'\}.
\]
Given $z\in \mathbb{C}$, define $\tilde z=\frac{z}{n'}$, $\hat z=z-tn'-2n'+r'$,   $ g=\gmult_{\D^L(\dig)}(\tilde z),$ and $ g'=\gmult_{\A(\dig')}(\hat z)$. Then
\[
\gmult_{\D^L(\dig\lexp\dig')}(z)=
\begin{array}{l}
\left\{\begin{aligned}
&ng' &\mbox{if }\; \hat z\in\spec_{\A}(\dig')\setminus\{r'\},\;\tilde z\not\in\spec_{\D^L}(\dig);\\
&g &\mbox{if }\;\hat z\not\in\spec_{\A}(\dig')\setminus\{r'\},\;\tilde z\in\spec_{\D^L}(\dig);\\
&ng'+g &\mbox{if }\;\hat z\in\spec_{\A}(\dig')\setminus\{r'\},\;\tilde z\in\spec_{\D^L}(\dig),\;ES_{\A(\dig')}(\hat z)\perp \bone_{n'} ;\\
&ng' &\mbox{if }\;\hat z\in\spec_{\A}(\dig')\setminus\{r'\},\;\tilde z\in\spec_{\D^L}(\dig),\;ES_{\A(\dig')}(\hat z)\not\perp \bone_{n'} ;\\
&0 &\mbox{otherwise}.
\end{aligned}\right.
\end{array}
\]
\end{corollary}
\begin{corollary}
\label{cor_spectrum_dist_SIGNLESS_lap_lex_doubly_dir}
Let $\dig$ and $\dig'$ be strongly connected digraphs of orders $n$ and $n'$, respectively, such that $\dig$ is $t$-transmission regular, every vertex of $\dig$ is incident with a doubly directed arc, and all vertices in $\dig'$ have out-degree $r'$. Let $\spec_{\D^Q}(\dig)=(\dQev_1,\dQev_2,\dots,\dQev_n)$ and $\spec_{\A}(\dig')=(r',\alpha'_2,\dots,\alpha'_{n'})$.  
Then
\[
\spec_{\D^Q}(\dig{\lexp}\dig')=\{n'\dQev_i+4n'-4-2r', \ i=1,\dots,n\}\cup\{(tn'+2n'-r'-\alpha'_j-4)^{(n)}, \ j=2,\dots,n'\}.
\]
Given $z\in \mathbb{C}$, define $\tilde z=\frac{z-4n'+4+2r'}{n'}$, $\hat z=tn'+2n'-r'-4-z$,  $ g=\gmult_{\D^Q(\dig)}(\tilde z),$ and $ g'=\gmult_{\A(\dig')}(\hat z)$. Then
\[
\gmult_{\D^Q(\dig\lexp\dig')}(z)=
\begin{array}{l}
\left\{\begin{aligned}
&ng' &\mbox{if }\; \hat z\in\spec_{\A}(\dig')\setminus\{r'\},\;\tilde z\not\in\spec_{\D^Q}(\dig);\\
&g &\mbox{if }\;\hat z\not\in\spec_{\A}(\dig')\setminus\{r'\},\;\tilde z\in\spec_{\D^Q}(\dig);\\
&ng'+g &\mbox{if }\;\hat z\in\spec_{\A}(\dig')\setminus\{r'\},\;\tilde z\in\spec_{\D^Q}(\dig),\;ES_{\A(\dig')}(\hat z)\perp \bone_{n'} ;\\
&ng' &\mbox{if }\;\hat z\in\spec_{\A}(\dig')\setminus\{r'\},\;\tilde z\in\spec_{\D^Q}(\dig),\;ES_{\A(\dig')}(\hat z)\not\perp \bone_{n'} ;\\
&0 &\mbox{otherwise}.
\end{aligned}\right.
\end{array}
\]
\end{corollary}

We next provide a description of the eigenvectors of $M \lexp M'$ from the eigenvectors of $M$ and $M'$, addressing the first two cases in Theorem \ref{thm_geometric_mult_lex_product_matrices}.

\begin{theorem}
\label{prop:evectors_lex_prod_general_matrices}
Let $M\in\mathbb{R}^{n\times n}$ and $M'\in\mathbb{R}^{n'\times n'}$ be irreducible nonnegative matrices, and suppose that $M'\bone_{n'}=\rho'\bone_{n'}$ for some $\rho'\in\mathbb{R}$. Let $\{\bv_1, \dots, \bv_k\}$ be a linearly independent set of eigenvectors %of $M$ 
with $M \bv_i = \lambda_i \bv_i$, %, \ \lambda_i \in \spec(M)$, 
and let $\{\bone_{n'},\bv_2', \dots, \bv_{k'}'\}$ be a linearly independent set of eigenvectors %of $M'$ 
with $M' \bv_{j}' = \lambda_j' \bv_j'$, %, \ \lambda_j' \in \spec(M')$. 
Then 
\ben[(1)]
\item \label{lexevec_1} For $i=1, \dots, k$, \, $\bv_i \otimes \bone_{n'}$ is  an eigenvector of $M \lexp M'$ corresponding to the eigenvalue  $n'\lambda_i + \rho'$.
\item \label{lexevec_2} For $j=2, \dots, k'$,  for $i=1, \dots, k$,  define $\gamma_{ij} = \frac{-\lambda_i{\bv_j'}^T\bone_{n'}}{\rho' + n'\lambda_i - \lambda_j'}$ when  $\lambda_j' \neq n'\lambda_i + \rho'$. Then
$\bv_i \otimes \bv_j' +\gamma_{ij}\bv_i \otimes \bone_{n'}$     is an eigenvector of $M\lexp M'$ for the eigenvalue  $\lambda_j'$.
\een
Furthermore, the set of eigenvectors of ${M}\lexp{M}'$ described in \eqref{lexevec_1} and \eqref{lexevec_2} is linearly independent. 
\end{theorem}

\begin{proof} First,  $(M\lexp M')(\bv_i \otimes \bone_{n'})  =   (M\otimes \J_{n'}+\I_n\otimes M')(\bv_i \otimes \bone_{n'})  
 %=   (M\otimes \J_{n'})(\bv_i \otimes \bone_{n'}) + (\I_n\otimes M')(\bv_i \otimes \bone_{n'}) 
 =  (M\bv_i)\otimes(\J_{n'} \bone_{n'}) + (\I_n\bv_i) \otimes (M'\bone_{n'})
 =  (\lambda_i\bv_i \otimes n'\bone_{n'}) + (\bv_i \otimes \rho'\bone_{n'})
 =  (n'\lambda_i + \rho')(\bv_i \otimes \bone_{n'})$. For the second statement,\vspace{-5pt}
\[\begin{aligned} (M\lexp M')(\bv_i \otimes \bv_j' +\gamma_{ij}\bv_i \otimes \bone_{n'}) &= \\
 (M\otimes \J_{n'}+\I_n\otimes M')(\bv_i \otimes \bv_j' +\gamma_{ij}\bv_i \otimes \bone_{n'}) &=\\
     (M\bv_i)\otimes(\J_{n'} \bv_j') + (\I_n\bv_i) \otimes (M'\bv_j') + \gamma_{ij} (M\bv_i)\otimes(\J_{n'} \bone_{n'}) + \gamma_{ij}(\I_n\bv_i) \otimes (M'\bone_{n'})&=\\
  \lambda_i {\bv_j'}^T\bone_{n'} (\bv_i \otimes \bone_{n'}) + \lambda_j'(\bv_i \otimes \bv_j') +  \gamma_{ij} \lambda_i n'(\bv_i \otimes \bone_{n'}) +  \gamma_{ij} \rho'(\bv_i \otimes \bone_{n'})&=\\
 \lambda_j' (\bv_i \otimes \bv_j' +\gamma_{ij}\bv_i \otimes \bone_{n'}).& \end{aligned}\] %=   (M\otimes \J_{n'})(\bv_i \otimes \bv_j') + (\I_n\otimes M')(\bv_i \otimes \bv_j')  + (M\otimes \J_{n'})(\gamma_{ij}\bv_i \otimes \bone_{n'}) + (\I_n\otimes M')(\gamma_{ij}\bv_i \otimes \bone_{n'}) 
 since $ - \lambda_j' \gamma_{ij} + {\lambda_i \bv_j'}^T\bone_{n'} + \gamma_{ij} \lambda_i n' + \gamma_{ij} \rho' = 0.$
 
The eigenvectors are linearly independent  by Lemma \ref{lem:kron-basis} and elementary linear algebra.  \end{proof}
In Corollaries \ref{cor_e_vec_adjacency_lex},  \ref{cor_e_vec_distance_lex_bigcycle}, and \ref{cor_e_vec_distance_lex_doubly_directed},
Theorem \ref{prop:evectors_lex_prod_general_matrices} is applied to provide a description of the eigenvectors of the adjacency and distance matrices of the lexicographic product of two digraphs.  Analogous results can be obtained for the (signless) Laplacian and for the (signless) distance Laplacian matrices with appropriate additional hypotheses by using analogous arguments.

\begin{corollary}
\label{cor_e_vec_adjacency_lex}
Let $\dig$ and $\dig'$ be strongly connected digraphs of orders $n$ and $n'$, respectively,  such that 
 $\dig'$ is $r'$-out-regular. Let $\{\bv_1, \dots, \bv_k\}$ be a linearly independent set of eigenvectors %of $\mathcal{A}(\dig)$ 
 with $\mathcal{A}(\dig) \bv_i = \alpha_i \bv_i$, % \ \alpha_i \in \spec_{\mathcal{A}}(\dig)$, 
 and let $\{\bone_{n'},\bv_2', \dots, \bv_{k'}'\}$ be a linearly independent set of eigenvectors %of $\mathcal{A}(\dig')$ 
 with $\mathcal{A}(\dig') \bv_{j}' = \alpha_j' \bv_j'.$ % \ \alpha_j' \in \spec_{\mathcal{A}}(\dig')$. 
 Then 
\ben[(1)]
\item \label{lexevec_1_adj_lex} For $i=1, \dots, k$, \, $\bv_i \otimes \bone_{n'}$ is  an eigenvector of $\mathcal{A}(\dig\lexp\dig')$ corresponding to the eigenvalue  $n'\alpha_i + r'$.
\item \label{lexevec_2_adj_lex} For $j=2, \dots, k'$,  for $i=1, \dots, k$, define $\gamma_{ij} = \frac{-\alpha_i{\bv_j'}^T\bone_{n'}}{r' + n'\alpha_i - \alpha_j'}$ when $\alpha_j' \neq n'\alpha_i + r'$. Then
$\bv_i \otimes \bv_j' +\gamma_{ij}\bv_i \otimes \bone_{n'}$     is an eigenvector of $\mathcal{A}(\dig\lexp\dig')$ for the eigenvalue  $\alpha_j'$.  \vspace{-5pt}
\een
Furthermore, the set of eigenvectors of $\mathcal{A}(\dig\lexp\dig')$ described in \eqref{lexevec_1_adj_lex} and \eqref{lexevec_2_adj_lex} is linearly independent.
\end{corollary}

\begin{corollary}
\label{cor_e_vec_distance_lex_bigcycle}
Let $\dig$ and $\dig'$ be strongly connected digraphs of orders $n$ and $n'$ %, respectively, 
such that $\dig'$ is $t'$-transmission regular and $\diam \dig' \leq \ds  g(\dig)$. Let $\{\bv_1, \dots, \bv_k\}$ be a linearly independent set of eigenvectors %of $\mathcal{D}(\dig)$ 
with $\mathcal{D}(\dig) \bv_i = \dev_i \bv_i,$ % \ \dev_i \in \spec_{\mathcal{D}}(\dig)$, 
and let $\{\bone_{n'},\bv_2', \dots, \bv_{k'}'\}$ be a linearly independent set of eigenvectors %of $\mathcal{D}(\dig')$ 
with $\mathcal{D}(\dig') \bv_{j}' = \dev_j' \bv_j'.$ % \ \dev_j' \in \spec_{\mathcal{D}}(\dig')$. 
Then 
\ben[(1)]
\item \label{lexevec_1_dist_lex_bigcycle} For $i=1, \dots, k$, \, $\bv_i \otimes \bone_{n'}$ is  an eigenvector of $\mathcal{D}(\dig\lexp\dig')$ corresponding to the eigenvalue  $n'\dev_i + t'$.
\item \label{lexevec_2_dist_lex_bigcycle} For $j=2, \dots, k'$, \ for $i=1, \dots, k$, define $\gamma_{ij} = \frac{-\dev_i{\bv_j'}^T\bone_{n'}}{t' + n'\dev_i - \dev_j'}$ when $\dev_j' \neq n'\dev_i + t'$.  Then 
$\bv_i \otimes \bv_j' +\gamma_{ij}\bv_i \otimes \bone_{n'}$   is an eigenvector of $\mathcal{D}(\dig\lexp\dig')$ for the eigenvalue  $\dev_j'$.\vspace{-5pt}\een
Furthermore, the set of eigenvectors of $\mathcal{D}(\dig\lexp\dig')$ described in \eqref{lexevec_1_dist_lex_bigcycle} and \eqref{lexevec_2_dist_lex_bigcycle} is linearly independent.
\end{corollary}

\begin{corollary}
\label{cor_e_vec_distance_lex_doubly_directed}
Let $\dig$ and $\dig'$ be strongly connected digraphs of orders $n$ and $n'$  such that  every vertex is incident with a doubly directed arc and all vertices in $\dig'$ have out-degree $r'$. Let $\{\bv_1, \dots, \bv_k\}$ be a linearly independent set of eigenvectors %of $\mathcal{D}(\dig)$ 
with $\mathcal{D}(\dig) \bv_i = \dev_i \bv_i$ % \ \dev_i \in \spec_{\mathcal{D}}(\dig)$, 
and let $\{\bone_{n'},\bv_2', \dots, \bv_{k'}'\}$ be a linearly independent set of eigenvectors %of $\mathcal{A}(\dig')$ 
with $\mathcal{A}(\dig') \bv_{j}' = \alpha_j' \bv_j'$. % \lambda_j' \in \spec_{\mathcal{A}}(\dig')$. 
Then 
\ben[(1)]
\item \label{lexevec_1_dist_lex_doubly_directed} For $i=1, \dots, k$, \, $\bv_i \otimes \bone_{n'}$ is  an eigenvector of $\mathcal{D}(\dig\lexp\dig')$ corresponding to the eigenvalue  $n'\dev_i + 2n'-2-r'$.
\item \label{lexevec_2_dist_lex_doubly_directed} For $j=2, \dots, k'$, \ for $i=1, \dots, k$, define $\beta_j=\frac{2\bv_j'^T\bone_{n'}}{r'-\alpha'_j-2n'}$ and $\gamma_{ij} = \frac{-\dev_i({\bv_j'}^T\bone_{n'}+n'\beta_j)}{2n'-r'+n'\dev_i+\alpha'_j}$ when $\alpha'_j\neq -2n'+r'-n'\dev_i$.  Then
$\bv_i \otimes \bv_j' +(\beta_j+\gamma_{ij})\bv_i \otimes \bone_{n'}$     is an eigenvector of $\mathcal{D}(\dig\lexp\dig')$ for the eigenvalue  $-\alpha'_j-2$.\vspace{-5pt}
\een
Furthermore, the set of eigenvectors of $\mathcal{D}(\dig\lexp\dig')$ described in \eqref{lexevec_1_dist_lex_doubly_directed} and \eqref{lexevec_2_dist_lex_doubly_directed} is linearly independent.
\end{corollary}

%\begin{proof} From Observation \ref{obs:lex-doubly-directed} we know that $\D(\dig \lexp \dig') = \D(\dig)\lexp M'$ where $M'=\A(\dig') + 2\A(\overline{\dig'})$. Proceeding as in the proof of Theorem \ref{thm:TRlexprod-dig_doubly_directed}, we observe that $\bone_{n'}$ is an eigenvector of $M'$ relative to the eigenvalue $2n'-2-r'$; moreover, for $j=2,\dots,k'$, $\bv_j'+\beta_j\bone_{n'}$ is an eigenvector of $M'$ relative to the eigenvalue $-\alpha'_j-2$. The result then follows by applying Theorem \ref{prop:evectors_lex_prod_general_matrices} and operating basic algebraic manipulations. \end{proof}
%

%
%%%%%%%%%%%%%%%%%%%%%%%%%%%%%%%%%%%%%%%%%%%%

\section{Direct products and strong products}\label{sDirectStrongprod}

For digraphs $\dig$ and $\dig'$,  $\A(\dig\directp\dig')=\A(\dig)\otimes \A(\dig')$  \cite{EH80} and $\A(\dig\strongp\dig')=\A(\dig\cp\dig') + \A(\dig\directp\dig')$; the formulas for graphs are analogous. The spectrum of the adjacency matrix of a direct product in terms of the  constituents is known:

\begin{theorem}{\rm \cite{EH80}}\label{thm:direct} Let $\dig$ and $\dig'$ be digraphs of orders $n$ and $n'$, respectively, having spectra 
$\spec_{\A}(\dig)=\{\alpha_1,\alpha_2,\dots,\alpha_n\}$ and $\spec_{\A}(\dig')=\{\alpha_1',\alpha'_2,\dots,\alpha'_{n'}\}$.  
Then
\[
\spec_{\A}(\dig{\directp}\dig')=\left\{\alpha_i\alpha_j' : i=1,\dots,n, \  j=1,\dots,n' \right\}.
\]
\end{theorem}

\begin{theorem}  Let $\dig$ and $\dig'$ be digraphs of orders $n$ and $n'$,  with 
$\spec_{\A}(\dig)=\{\alpha_1,\alpha_2,\dots,\alpha_n\}$ and $\spec_{\A}(\dig')=\{\alpha_1',\alpha'_2,\dots,\alpha'_{n'}\}$.  
Then
\[
\spec_{\A}(\dig{\strongp}\dig')=\left\{\alpha_i\alpha_j'+\alpha_i+\alpha_j' : i=1,\dots,n, \  j=1,\dots,n' \right\}.
\]
\end{theorem}

\bpf
Choose $C$ and $C'$ such that $C^{-1}\A(\dig)C=\JCF_{\A(\dig)}$ and $C'^{-1}\A(\dig')C'=\JCF_{\A(\dig')}$. Consider \[(C^{-1}\otimes C'^{-1})\A(\dig\strongp\dig')(C\otimes C')=(C^{-1}\otimes C'^{-1})\A(\dig\cp\dig')(C\otimes C')+(C^{-1}\otimes C'^{-1})\A(\dig\directp\dig')(C\otimes C').\]
As  in the proof of \cite[Theorem 4.4.5]{HJ2}, $(C^{-1}\otimes C'^{-1})\A(\dig\cp\dig')(C\otimes C')$ is an upper triangular matrix with diagonal entries $\left\{\alpha_i+\alpha_j' : i=1,\dots,n, \  j=1,\dots,n' \right\}.$ The proof of Theorem \ref{thm:direct}, which utilizes a result from Lancaster \cite[p. 259-260]{L69}, shows $(C^{-1}\otimes C'^{-1})\A(\dig\directp\dig')(C\otimes C')$ is an upper triangular matrix with diagonal entries $\left\{\alpha_i\alpha_j' : i=1,\dots,n, \  j=1,\dots,n' \right\}.$
Therefore $(C^{-1}\otimes C'^{-1})\A(\dig\strongp\dig')(C\otimes C')$ is an upper triangular matrix with diagonal entries $\left\{\alpha_i\alpha_j'+\alpha_i+\alpha_j' : i=1,\dots,n, \  j=1,\dots,n' \right\}$. %The diagonal entries of an upper triangular matrix are its eigenvalues so $\spec_{\A}(\dig{\strongp}\dig')=\left\{\alpha_i\alpha_j'+\alpha_i+\alpha_j' : i=1,\dots,n, \  j=1,\dots,n' \right\}$ as desired.
\epf

Since the direct product of strongly connected digraphs is not necessarily strongly connected,  the distance matrix may be undefined. However,  the strong product of  strongly connected digraphs is  strongly connected,  and  the following distance formula is known.

\begin{proposition}\label{diststrong} {\rm \cite[Proposition 10.2.1]{H18}}
Let $\dig$ and $\dig'$ be strongly connected digraphs. Then the  distance formula for the strong product $\dig \strongp \dig'$ is
\[
d_{\dig \strongp \dig'}((x,x'),(y,y')) = \max \{d_{\dig} (x,y), d_{\dig'} (x',y')\}.
\]
\end{proposition}

Given this formula for distance, the methods developed here do not seem to be applicable to determining the spectra of distance matrices of strong products of digraphs.
%
%%%%%%%%%%%%%%%%%%%%%%%%%%%%%%%%%%%%%%%%%%%%

\section{Directed strongly regular graphs}\label{sSRD}
In this section we discuss directed strongly regular graphs (DSRGs), a special class of digraphs all of which have diameter at most two and are {\em regular}, meaning all vertices have in-degree and out-degree equal to some common value $k$; such a digraph is also called {\em $k$-regular}. A DSRG requires additional properties, and it is noteworthy that a DSRG  has exactly three distinct eigenvalues; we apply our Cartesian product formula to a DSRG to produce an infinite family of graphs with three distinct eigenvalues.

Before defining a DSRG, we first prove a more general result about $k$-regular digraphs with diameter at most two, which is analogous to a result for graphs.  Note that any such digraph of order $n$ is transmission regular with transmission $2n-2-k$.  %result in \cite{EGM04} for graphs. 

%If is a $k$-regular digraph  of order $n$ and diameter at most $2$, then $\D(\dig)= \A(\dig)+2\A(\overline{\dig})$. The next results is now immediate from Proposition \ref{o:adj-comp}.

\begin{proposition}\label{Diam2Prop}
Let $\dig$ be a $k$-regular digraph  of order $n$ and diameter at most $2$ with $\spec_\A(\dig)=\{k,\alpha_2,\dots,\alpha_{n}\}$.  Then $\spec_\D(\dig)=\{2n-2-k,-(\alpha_2+2),\dots,-(\alpha_{n}+2)\}$, $\bone_n$ is an eigenvector  of $\D(\dig)$ for eigenvalue $2n-2-k$,  and if $\bv_i$ is an eigenvector of $\A(\dig)$ for $\alpha_i\ne k$, then $\bv_i$ is an eigenvector of $\D(\dig)$  for $-2-\alpha_i$.  Furthermore, $\gmult_{\D(\dig)}(-\alpha_i-2)=\gmult_{\A(\dig)}(\alpha_i)$ for $\alpha_i\ne k$ and    $\gmult_{\D(\dig)}(-k-2)=\gmult_{\A(\dig)}(k)-1$. \end{proposition}

\begin{proof}
Because $\D(\dig)=\A(\dig)+2\A(\overline{\dig})$, all the statements except the geometric multiplicity of eigenvalue $-k-2$ of $\D(\dig)$ will follow from Proposition \ref{o:adj-comp} once we show that $\bone^T\bv_i=0$  for $\alpha_i\ne k$.
%Let $\A=\A(\dig)$,  $\D=\D(\dig)$, $\J=\J_n$, and $\bone=\bone_n$.  %We establish the statements about eigenvectors, and note the statements about eigenvalues then follow. 
Since $\dig$ is $k$-regular, %it is immediate that $\A(\dig)\bone_n=k\bone_n$ and $\bone_n^T \A(\dig)=k\bone_n^T$. This implies  
$\A(\dig)\J_n=k\J_n=\J_n \A(\dig)$. Let $\bone_n^T\bv_i=c_i$,  so $\J_n\bv_i=c_i\bone_n$.  Then\vspace{-3pt}
\[c_ik\bone_n=c_i\A(\dig)\bone_n=\A(\dig)\J_n\bv_i=\J_n\A(\dig)\bv_i=\J_n\alpha_i\bv_i=c_i\alpha_i\bone_n.\vspace{-5pt}\]
Since $k\ne \alpha_i$, this implies $c_i=0$.  To see that $\gmult_{\D(\dig)}(-k-2)=\gmult_{\A(\dig)}(k)-1$,  choose an orthogonal basis of eigenvectors for $ES_{\A(\dig)}(k)$ that includes $\bone_n$.   %{\red Therefore, every eigenvector of $\A(\dig)$ is an eigenvector of $\J_n$, so every eigenvector that is not a multiple of $\bone_n$ is orthogonal to $\bone_n$.}
\end{proof}

%It is important to note that in general, this proposition does not recover multiplicities since there is no guarantee that $\A$ has a basis of eigenvectors.

Strongly regular graphs are a well studied family of graphs which are of particular interest because they have exactly three eigenvalues. Duval \cite{D88} defined a {\em directed strongly regular graph}, here denoted by $\dig(n,k,s,a,c)$, to be a digraph $\dig$ of order $n$ such that \vspace{-3pt}
\[\A(\dig)^2=s\I_n+a \A(\dig) + c (\J_n-\I_n-\A(\dig)) \text{  and  } \A(\dig)\J_n=\J_n\A(\dig)=k\J_n. \vspace{-3pt}\] 
Such a digraph is $k$-regular and each vertex is incident with $s$   doubly directed arcs. The number of directed paths of length two from a vertex $v$ to a vertex $u$ is $a$ if $(v,u)$ is an arc in $\dig$ and $c$ if $(v,u)$ is not an arc in $\dig$. Duval originally used the notation $\dig(n,k,\mu,\lambda,t)$ where $\lam=a$, $\mu=c$, and $t=s$ in our notation. We use $s$ rather than $t$ to follow the distance matrix literature in using $t$ for transmission. Both usages $G(n,k,a,c)$ and $G(n,k,\lambda,\mu)$ appear in the literature for strongly regular graphs, and we avoid using $\lambda$ since it has been used throughout this paper as an eigenvalue.  The reordering $\dig(n,k,t,\lambda,\mu)$ of Duval's original notation $\dig(n,k,\mu,\lambda,t)$ has become popular in more recent literature since it more closely follows the standard ordering for strongly regular graphs.

Duval computed the next formula for the eigenvalues of $\A(\dig(n,k,s,a,c))$. 

\begin{theorem}{\rm \cite{D88}}\label{t:DSRG-specA}
%The three   eigenvalues of $\A(\dig)$ for  $\dig=\dig(n,k,s,\lambda,\mu)$ are
Let $\dig=\dig(n,k,s,a,c)$. The spectrum of $\A(\dig)$ consists of the three   eigenvalues 
\[\theta_1=k,\ \theta_2=\frac{1}{2}\left(a -c + \sqrt{(c-a)^2+4(s-c)} \right)\!,\mbox{ and } \theta_3=\frac{1}{2}\left( a -c - \sqrt{(c-a)^2+4(s-c)} \right)\vspace{-3pt} \]
with multiplicities \vspace{-3pt}
\[\mult(\theta_1)=1,\ \mult(\theta_2)=-\frac{k+\theta_3(n-1)}{\theta_2-\theta_3}, \text{ and } \mult(\theta_3)=\frac{k+\theta_2(n-1)}{\theta_2-\theta_3}.\vspace{-5pt} \]
\end{theorem}

%In \cite{J03}, J\o rgensen proves that the adjacency matrix of every DSRG is diagonalizable and thus has a basis of eigenvectors. Note that this property does not hold for all transmission regular digraphs of diameter at most 2:   Figure \ref{fig:TRegNoEvec} is an example of a digraph $\dig$ that does not have a basis of eigenvectors; note that the digraph obtained from $\dig$ is not transmission regular, {\red whereas reversing every arc in a DSRG produces a DSRG}. his property allows us to apply 

Duval's theorem and Proposition \ref{Diam2Prop} determine the $\D$-spectrum of a direct strongly regular graph.

\begin{corollary}
Let $\dig=\dig(n,k,s,a,c)$. The spectrum of $\D(\dig)$ consists of the three   eigenvalues
{\scriptsize\[\partial_1=2n-2-k,\, \partial_2=-2-\frac{1}{2}\left(a -c + \sqrt{(c-a)^2+4(s-c)} \right)\!,\,\mbox{and } \partial_3=-2-\frac{1}{2}\left( a -c - \sqrt{(c-a)^2+4(s-c)} \right) \]}
with multiplicities $\mult(\partial_i)=\mult(\theta_i)$ for $i=1,2,3$.
\end{corollary}

In \cite{J03}, J\o rgensen proved that the adjacency matrix of every DSRG is diagonalizable and thus has a basis of eigenvectors. By Proposition \ref{Diam2Prop}, this property is also true of the distance matrix of a DSRG.  Note that this property does not hold for all transmission regular digraphs of diameter at most 2:   Figure \ref{fig:TRegNoEvec} is an example of a digraph $\dig$ that does not have a basis of eigenvectors; note that the digraph obtained from $\dig$ by reversing every arc is not transmission regular,  whereas reversing every arc in a DSRG produces a DSRG.

Cartesian products provide a method of forming digraphs on a large number of vertices with few distinct distance eigenvalues. Applying Theorem \ref{thm:TRcartprod-dig_new_new} to transmission regular digraphs $\dig$ on $n$ vertices and $\dig '$ on $n'$ vertices, we see that $\dig\cp\dig '$ has $nn'$ vertices but at most $n+n'$ distinct eigenvalues. The number of distinct eigenvalues can be much lower if the spectra of $\dig$ and $\dig'$ share some common values or if they contain $0$ as an eigenvalue. \

\begin{proposition}\label{p:cp-DSRG} Suppose $\dig$ is a transmission regular digraph of order $n$ with $\spec_\D(\dig)=\{t=\partial_1, \partial_2^{(m)}, 0^{(n-1-m)}\}$.
 Define $\dig_{\ell}=\dig\cp\dots\cp\dig$, the Cartesian product of $\ell$ copies of $\dig$. Then the order of $\dig_\ell$ is $n^\ell$ and  $\spec_{\D}(\dig_{\ell})=\{\ell t\,n^{\ell-1},\left(\partial_2\,n^{\ell-1}\right)^{(m\ell)},0^{(n^\ell-1-m\ell)}\}$.
\end{proposition}

\begin{proof}
We prove the claim by induction. When $\ell=2$, Theorem \ref{thm:TRcartprod-dig_new_new} implies $\spec_{\D}(\dig_{2})=\{2nt,\left(\partial_2\,n\right)^{(2m)},$ $0^{(n^2-1-2m)}\}$. %, so the base case is true. 
Now assume $\spec_{\D}(\dig_{\ell})=\{\ell t\,n^{\ell-1},\left(\partial_2\,n^{\ell-1}\right)^{(m\ell)},0^{(n^\ell-1-m\ell)}\}$.  Since $\dig_{\ell+1}=\dig_\ell\cp \dig$, applying Theorem \ref{thm:TRcartprod-dig_new_new} again we get
\[\begin{aligned}
\spec_{\D}(\dig_{\ell+1})&=\{n\,\ell tn^{\ell-1}+n^{\ell}t, \left(n\partial_2\,n^{\ell-1}\right)^{(m\ell)},0^{(n^\ell-1-m\ell)},\left(n^\ell\partial_2\right)^{(m)},0^{(n-1-m)},0^{(n^\ell-1)(n-1)}\}\\
%&=\{10k(8^{k})+10(8^{k}), \left(-2(8^{k})\right)^{(5k)},0^{(8^{k}-1-5k)},\left(-2(8^k)\right)^{(5)},0^{(2)},0^{(8^{k+1}-8^k-7)}\}\\
&=\{t(\ell+1)n^{\ell}, \left(\partial_2\,n^{\ell}\right)^{(m(\ell+1))},0^{\left(n^{\ell+1}-1-m(\ell+1)\right)}\}.
\end{aligned}\]
\end{proof}

\begin{example}{\rm
The DSRG $\dig=\dig(8,4,3,1,3)$ has spectrum $\spec_{\D}(\dig)=\{10,-2^{(5)},0^{(2)}\}$. Therefore this digraph allows us to construct examples of arbitrarily large digraphs with only three distinct eigenvalues.   By Proposition \ref{p:cp-DSRG}, $\dig_\ell$ has order $8^\ell$ and   $\spec_{\D}(\dig_{\ell})=\{10\ell(8^{\ell-1}),\left(-2(8^{\ell-1})\right)^{(5\ell)},0^{(8^{\ell}-1-5\ell)}\}$.}
\end{example}

\begin{figure}[h!]
    \centering
\scalebox{.7}{    \begin{tikzpicture}
\definecolor{cv0}{rgb}{0.0,0.0,0.0}
\definecolor{cfv0}{rgb}{1.0,1.0,1.0}
\definecolor{clv0}{rgb}{0.0,0.0,0.0}
\definecolor{cv1}{rgb}{0.0,0.0,0.0}
\definecolor{cfv1}{rgb}{1.0,1.0,1.0}
\definecolor{clv1}{rgb}{0.0,0.0,0.0}
\definecolor{cv2}{rgb}{0.0,0.0,0.0}
\definecolor{cfv2}{rgb}{1.0,1.0,1.0}
\definecolor{clv2}{rgb}{0.0,0.0,0.0}
\definecolor{cv3}{rgb}{0.0,0.0,0.0}
\definecolor{cfv3}{rgb}{1.0,1.0,1.0}
\definecolor{clv3}{rgb}{0.0,0.0,0.0}
\definecolor{cv4}{rgb}{0.0,0.0,0.0}
\definecolor{cfv4}{rgb}{1.0,1.0,1.0}
\definecolor{clv4}{rgb}{0.0,0.0,0.0}
\definecolor{cv5}{rgb}{0.0,0.0,0.0}
\definecolor{cfv5}{rgb}{1.0,1.0,1.0}
\definecolor{clv5}{rgb}{0.0,0.0,0.0}
\definecolor{cv6}{rgb}{0.0,0.0,0.0}
\definecolor{cfv6}{rgb}{1.0,1.0,1.0}
\definecolor{clv6}{rgb}{0.0,0.0,0.0}
\definecolor{cv7}{rgb}{0.0,0.0,0.0}
\definecolor{cfv7}{rgb}{1.0,1.0,1.0}
\definecolor{clv7}{rgb}{0.0,0.0,0.0}
\definecolor{cv0v1}{rgb}{0.0,0.0,0.0}
\definecolor{cv0v2}{rgb}{0.0,0.0,0.0}
\definecolor{cv0v5}{rgb}{0.0,0.0,0.0}
\definecolor{cv0v6}{rgb}{0.0,0.0,0.0}
\definecolor{cv1v2}{rgb}{0.0,0.0,0.0}
\definecolor{cv1v3}{rgb}{0.0,0.0,0.0}
\definecolor{cv1v4}{rgb}{0.0,0.0,0.0}
\definecolor{cv1v5}{rgb}{0.0,0.0,0.0}
\definecolor{cv2v0}{rgb}{0.0,0.0,0.0}
\definecolor{cv2v3}{rgb}{0.0,0.0,0.0}
\definecolor{cv2v4}{rgb}{0.0,0.0,0.0}
\definecolor{cv2v7}{rgb}{0.0,0.0,0.0}
\definecolor{cv3v0}{rgb}{0.0,0.0,0.0}
\definecolor{cv3v1}{rgb}{0.0,0.0,0.0}
\definecolor{cv3v6}{rgb}{0.0,0.0,0.0}
\definecolor{cv3v7}{rgb}{0.0,0.0,0.0}
\definecolor{cv4v1}{rgb}{0.0,0.0,0.0}
\definecolor{cv4v2}{rgb}{0.0,0.0,0.0}
\definecolor{cv4v5}{rgb}{0.0,0.0,0.0}
\definecolor{cv4v6}{rgb}{0.0,0.0,0.0}
\definecolor{cv5v0}{rgb}{0.0,0.0,0.0}
\definecolor{cv5v1}{rgb}{0.0,0.0,0.0}
\definecolor{cv5v6}{rgb}{0.0,0.0,0.0}
\definecolor{cv5v7}{rgb}{0.0,0.0,0.0}
\definecolor{cv6v0}{rgb}{0.0,0.0,0.0}
\definecolor{cv6v3}{rgb}{0.0,0.0,0.0}
\definecolor{cv6v4}{rgb}{0.0,0.0,0.0}
\definecolor{cv6v7}{rgb}{0.0,0.0,0.0}
\definecolor{cv7v2}{rgb}{0.0,0.0,0.0}
\definecolor{cv7v3}{rgb}{0.0,0.0,0.0}
\definecolor{cv7v4}{rgb}{0.0,0.0,0.0}
\definecolor{cv7v5}{rgb}{0.0,0.0,0.0}
\Vertex[style={minimum
size=1.0cm,draw=cv0,fill=cfv0,text=clv0,shape=circle},LabelOut=false,L=\hbox{$1$},x=3cm,y=4.5cm]{v0}
\Vertex[style={minimum
size=1.0cm,draw=cv1,fill=cfv1,text=clv1,shape=circle},LabelOut=false,L=\hbox{$2$},x=4.5cm,y=3cm]{v1}
\Vertex[style={minimum
size=1.0cm,draw=cv2,fill=cfv2,text=clv2,shape=circle},LabelOut=false,L=\hbox{$3$},x=4.5cm,y=1.5cm]{v2}
\Vertex[style={minimum
size=1.0cm,draw=cv3,fill=cfv3,text=clv3,shape=circle},LabelOut=false,L=\hbox{$4$},x=3cm,y=0cm]{v3}
\Vertex[style={minimum
size=1.0cm,draw=cv4,fill=cfv4,text=clv4,shape=circle},LabelOut=false,L=\hbox{$5$},x=1.5cm,y=0cm]{v4}
\Vertex[style={minimum
size=1.0cm,draw=cv5,fill=cfv5,text=clv5,shape=circle},LabelOut=false,L=\hbox{$6$},x=0cm,y=1.5cm]{v5}
\Vertex[style={minimum
size=1.0cm,draw=cv6,fill=cfv6,text=clv6,shape=circle},LabelOut=false,L=\hbox{$7$},x=0cm,y=3cm]{v6}
\Vertex[style={minimum
size=1.0cm,draw=cv7,fill=cfv7,text=clv7,shape=circle},LabelOut=false,L=\hbox{$8$},x=1.5cm,y=4.5cm]{v7}
\Edge[lw=0.1cm,style={post, color=cv0v1,},](v0)(v1)
\Edge[lw=0.1cm,style={color=cv0v2,},](v0)(v2)
\Edge[lw=0.1cm,style={color=cv0v2,},](v0)(v5)
\Edge[lw=0.1cm,style={color=cv0v2,},](v0)(v6)
\Edge[lw=0.1cm,style={post,color=cv1v2,},](v1)(v2)
\Edge[lw=0.1cm,style={color=cv0v2,},](v1)(v3)
\Edge[lw=0.1cm,style={color=cv0v2,},](v1)(v4)
\Edge[lw=0.1cm,style={color=cv0v2,},](v1)(v5)
% \Edge[lw=0.1cm,style={post, bend right,color=cv2v0,},](v2)(v0)
\Edge[lw=0.1cm,style={post,color=cv2v3,},](v2)(v3)
\Edge[lw=0.1cm,style={color=cv0v2,},](v2)(v4)
\Edge[lw=0.1cm,style={color=cv0v2,},](v2)(v7)
\Edge[lw=0.1cm,style={post,,color=cv3v0,},](v3)(v0)
% \Edge[lw=0.1cm,style={post, bend right,color=cv3v1,},](v3)(v1)
\Edge[lw=0.1cm,style={color=cv0v2,},](v3)(v6)
\Edge[lw=0.1cm,style={color=cv0v2,},](v3)(v7)
% \Edge[lw=0.1cm,style={post, bend right,color=cv4v1,},](v4)(v1)
% \Edge[lw=0.1cm,style={post, bend right,color=cv4v2,},](v4)(v2)
\Edge[lw=0.1cm,style={post,color=cv4v5,},](v4)(v5)
\Edge[lw=0.1cm,style={color=cv0v2,},](v4)(v6)
% \Edge[lw=0.1cm,style={post, bend right,color=cv5v0,},](v5)(v0)
% \Edge[lw=0.1cm,style={post, bend right,color=cv5v1,},](v5)(v1)
\Edge[lw=0.1cm,style={post,color=cv5v6,},](v5)(v6)
\Edge[lw=0.1cm,style={color=cv0v2,},](v5)(v7)
% \Edge[lw=0.1cm,style={post, bend right,color=cv6v0,},](v6)(v0)
% \Edge[lw=0.1cm,style={post, bend right,color=cv6v3,},](v6)(v3)
% \Edge[lw=0.1cm,style={post, bend right,color=cv6v4,},](v6)(v4)
\Edge[lw=0.1cm,style={post,color=cv6v7,},](v6)(v7)
% \Edge[lw=0.1cm,style={post, bend right,color=cv7v2,},](v7)(v2)
% \Edge[lw=0.1cm,style={post, bend right,color=cv7v3,},](v7)(v3)
\Edge[lw=0.1cm,style={post,color=cv7v4,},](v7)(v4)
% \Edge[lw=0.1cm,style={post, bend right,color=cv7v5,},](v7)(v5)
%
\end{tikzpicture}}
    \caption{$\dig(8,4,3,1,3)$}
    \label{fig:my_label}
\end{figure}

Because directed strongly regular graphs are transmission regular,   the $\DL$ and $\DQ$ eigenvalues of directed strongly regular graphs are immediate.

\begin{corollary}
Let $\dig=\dig(n,k,s,a,c)$. The spectrum of $\DL(\dig)$ consists of the three   eigenvalues
{\scriptsize \[\partial^L_1=0,\ \partial^L_2=2n-k+\frac{1}{2}\left( a -c + \sqrt{(c-a)^2+4(s-c)} \right)\!, \mbox{ and }\partial^L_3=2n-k+\frac{1}{2}\left( a -c - \sqrt{(c-a)^2+4(s-c)} \right) \]}
with multiplicities $\mult(\partial^L_i)=\mult(\theta_i)$ for $i=1,2,3$.
The spectrum of $\DQ(\dig)$ consists of the three   eigenvalues
{\scriptsize\[ \partial^Q_1=4n-4-2k,\ \partial^Q_2=2n-k-4-\frac{1}{2}\left( a -c + \sqrt{(c-a)^2+4(s-c)} \right)\!, \mbox{ and }\partial^Q_3=2n-k-4-\frac{1}{2}\left( a -c - \sqrt{(c-a)^2+4(s-c)} \right) \]}
with multiplicities $\mult(\partial^Q_i)=\mult(\theta_i)$ for $i=1,2,3$.
\end{corollary}

Because  directed strongly regular graphs are out-regular,   the Laplacian and signless Laplacian eigenvalues of directed strongly regular graphs are also immediate from Theorem \ref{t:DSRG-specA}.

While the eigenvalues for $\A(\dig)$, $L(\dig)$, $Q(\dig)$, $\D(\dig)$, $\DL(\dig)$, and $\DQ(\dig)$ can be non-real, this is not true for most DSRGs. For a DSRG that is not equivalent to a graph and is not a doubly regular tournament $\dig(2k+1,k,0,a,a+1)$, Duval proved $(c-a)^2+4(s-c)=d^2$ for some positive integer $d$, which implies all eigenvalues of $\A(\dig)$, $L(\dig)$, $Q(\dig)$, $\D(\dig)$, $\DL(\dig)$, and $\DQ(\dig)$ are rational. In the case of graphs, it is well known that these spectra are real. Before we consider the only remaining case, we need the following lemma from Klin et al.

\begin{lemma}{\rm \cite{KMMZ04}}
Let $\dig$ be a regular non-empty digraph without doubly directed arcs. Then $\A(\dig)$ has at least one non-real eigenvalue.
\end{lemma}

Applying the previous lemma, we obtain the next result about instances of non-real  eigenvalues in a DSRG. 

\begin{corollary}
For the DSRG $\dig=\dig(n,k,s,a,c)$, the spectra of $\A(\dig)$, $L(\dig)$, $Q(\dig)$, $\D(\dig)$, $\DL(\dig)$, and $\DQ(\dig)$ contain non-real eigenvalues if and only if $\dig= \dig(2k+1,k,0,a,a+1)$.
\end{corollary}
% \begin{proof}
% It has already been shown that no other DSRG can have complex eigenvalues. Let $\dig= \dig(2k+1,k,0,\lambda,\lambda+1)$ and consider the discriminant
% \[ (\mu-\lambda)^2+4(s-\mu)= 1+4(-\lambda-1)=-3-4\lambda.\]
% Since $\lambda$ is a non-negative integer, the discriminant is always negative in this case, resulting in complex eigenvalues.
% \end{proof}

%\subsection{An infinite family of digraphs with 3 distinct eigenvalues}

\bigskip
{\bf Acknowledgment.} The research of Minerva Catral was supported by a Faculty Development Leave from Xavier University. The research of Carolyn Reinhart was supported by NSF DMS 1839918.

%%%%%%%%%%%%%%%%%%%%%%%%%%%%%%%%%%%%%%%%%%%

\end{document}